\documentclass{amsart}

\usepackage{amsmath,amsthm,amssymb}
\usepackage{enumerate}
\usepackage{graphicx}
\usepackage{float}

\theoremstyle{plain}
\newtheorem{thm}{Theorem}[section]
\newtheorem{cor}[thm]{Corollary}
\newtheorem{lem}[thm]{Lemma}
\newtheorem{prop}[thm]{Proposition}

\theoremstyle{definition}
\newtheorem{defn}[thm]{Definition}
\theoremstyle{remark}

\usepackage{tikz}
\usetikzlibrary{arrows,shapes,trees}
\usepackage{graphicx}
\usepackage{hyperref}
\usepackage{amssymb}
\usepackage{amsfonts}
\usepackage{amsthm}
\usepackage{amsmath}
\usepackage{epstopdf}
\usepackage{mathrsfs}
\numberwithin{equation}{section}

\newcommand{\N}{\mathbb N}

\newcommand{\Z}{\mathbb Z}

\newcommand{\R}{\mathbb R}

\newcommand{\E}{\mathcal{E}}

\newcommand{\e}{\varepsilon}

\renewcommand{\l}{\ell}
\newcommand{\la}{\lambda}

\title[Measurable rigidity of the cohomological equation]{Measurable rigidity of the cohomological equation for linear cocycles over hyperbolic systems}
\author{Clark Butler}
\thanks{This material is based upon work supported by the National Science Foundation Graduate Research Fellowship under Grant \# DGE-1144082. }
\begin{document}

\begin{abstract}
We show that any measurable solution of the cohomological equation for a H\"older  linear cocycle over a hyperbolic system coincides almost everywhere with a H\"older  solution.  More generally, we show that every measurable invariant conformal structure for a H\"older linear cocycle over a hyperbolic system coincides almost everywhere with a continuous invariant conformal structure. We also use the main theorem to show that a linear cocycle is conformal if none of its iterates preserve a measurable family of proper subspaces of $\mathbb{R}^{d}$. We use this to characterize closed negatively curved Riemannian manifolds of constant negative curvature by irreducibility of the action of the geodesic flow on the unstable bundle. 
\end{abstract}
\maketitle

\section{Introduction}
This work is motivated by a pair of questions, one of which is sourced from the geometry of negatively curved manifolds and the other from the study of the cohomology of linear cocycles over hyperbolic systems. We will first explain how the central difficulty of these two motivating questions can be viewed as two different manifestations of the same problem regarding continuity of measurable invariant conformal structures for linear cocycles over hyperbolic systems. 

The geometric question was originally posed by Sullivan \cite{Su}, who conjectured that if the action of the geodesic flow of a closed negatively curved Riemannian manifold $M$ (with $\dim M \geq 3$) on the tangent spaces to expanding horospheres is \emph{measurably irreducible} and the sectional curvatures of $M$ satisfy the $1/4$-pinching condition $-\frac{1}{4} > K \geq -1$ then $M$ has constant negative curvature. We let $E^{u}$ denote the unstable bundle of the geodesic flow $g^{t}$ on $T^{1}M$ which is tangent to the expanding horospheres of $M$ and we let $m$ denote the invariant Liouville volume for the geodesic flow. Measurable irreducibility in this context then means that there is no measurable family of proper subspaces of $E^{u}$ which are invariant $m$-a.e. under the derivative action of the geodesic flow. This conjecture was confirmed by combining the work of Kanai \cite{Kan1} and Besson-Courtois-Gallot \cite{BCG1}, as noted by Yue \cite{Yue}.

The $1/4$-pinching assumption on the sectional curvatures of $M$ appears superfluous at first glance in the hypotheses of this theorem, as one does not require $1/4$-pinching to formulate the notion of measurable irreducibility of the geodesic flow on expanding horospheres, and the conclusion of the theorem shows that the curvature of the manifold is actually constant. This was also observed by Yue, who proposed a proof that the measurable irreducibility hypothesis alone suffices to conclude that the curvature of $M$ is constant \cite{Yue}. However there appears to be a crucial gap in this proof which is discussed in detail in \cite[Remark 3.2]{Bu1}. In our first theorem we close this gap and remove the $1/4$-pinching hypothesis,  

\begin{thm}\label{hyperbolic rigidity}
Let $g^{t}$ be the geodesic flow of a closed negatively curved Riemannian manifold $M$ with $\dim M \geq 3$. If $Dg^{t}|_{E^{u}}$ is measurably irreducible then $M$ has constant negative curvature. 
\end{thm}

Our second motivating question comes from the study of the cohomology of linear cocycles over hyperbolic systems. We will take subshifts of finite type as our models for hyperbolic systems. For this we recall some standard definitions. Let $\l \geq 2$ be a fixed positive integer, let $Q = (q_{ij})_{1 \leq i,j \leq r }$ be an $\l \times \l$ matrix with $q_{ij} \in \{0,1\}$ and let $\Sigma$ be the \emph{subshift of finite type} associated to the matrix $Q$,
\[
\Sigma = \{(x_{n})_{n \in \mathbb{Z}}: q_{x_{n}x_{n+1}} = 1 \;\text{for all $n \in \mathbb{Z}$}\}. 
\] 
We let $f: \Sigma\circlearrowleft$ be the left shift map defined by $f(x_{n})_{n \in \mathbb{Z}} = (x_{n+1})_{n \in \mathbb{Z}}$. We equip $\Sigma$ with the family of metrics,
\begin{equation}\label{metric def}
\rho_{\tau}(x,y)=e^{-\tau N(x,y)},\; \textrm{where} \; N(x,y)=\max \lbrace N\geq 0; x_n=y_n \; \textrm{for all} \;|n|<N \rbrace,
\end{equation}
where $\tau \in (0,\infty)$. Let $d \geq 1$ be a fixed positive integer.
\begin{defn}\label{defn: cocycle}
Let $A: \Sigma \rightarrow GL(d,\R)$ be a measurable map. The \emph{linear cocycle} over $f$ generated by $A$ is the map
\[
\mathcal{A}: \Sigma \times \Z \rightarrow GL(d,\R),
\]
\[
\mathcal{A}(x,n) = \left\{
	\begin{array}{ll}
		A(f^{n-1}(x))\ldots A(f(x))A(x)  & \mbox{if } n>0 \\
		I & \mbox{if } n=0 \\
		A(f^{n}(x))^{-1}\ldots A(f^{-1}(x))^{-1}& \mbox{if } n<0. \\
	\end{array}
\right.
\]
\end{defn}
We will write $A^{n}(x):=\mathcal{A}(x,n)$. We say that $\mathcal{A}$ is a continuous linear cocycle if the generator $A$ is continuous; we say that $\mathcal{A}$ is $\alpha$-H\"older continuous if the generator $A$ is $\alpha$-H\"older continuous with respect to the metric $\rho_{\tau}$ for a fixed choice of $\tau$. Note that it then follows that $A$ is H\"older continuous with respect to $\rho_{\tau'}$ for any other $\tau' > 0$, with a possibly different H\"older exponent. 

Lastly, let $\mu$ be a fixed ergodic, fully supported $f$-invariant probability measure on $\Sigma$ with local product structure. Local product structure will be defined in Section \ref{shift definitions}; natural examples of measures with local product structure are equilibrium states of H\"older potentials \cite{Bow}, such as the measure of maximal entropy. 

\begin{defn}\label{defn: cohomologous} Two linear cocycles $A$ and $B$ over $f$ are (measurably, continuously) \emph{cohomologous} if there is a (measurable, continuous) function $P: \Sigma \rightarrow GL(d,\R)$ such that
\[
A(x) = P(f(x))B(x)P(x)^{-1},
\]
for $\mu$-a.e. $x \in \Sigma$ in the measurable case, and every $x \in \Sigma$ in the continuous case. 
\end{defn}

Our second main theorem states that if a H\"older continuous linear cocycle $\mathcal{A}$ over $f$ is measurably cohomologous to the identity cocycle $B(x) \equiv  Id_{\R^{d}}$ then $A$ is continuously cohomologous to the identity.
 
\begin{thm}\label{coboundary}
Let $\mathcal{A}$ be an $\alpha$-H\"older continuous cocycle over $f$.   Suppose that there is a measurable function $P : \Sigma \rightarrow GL(d,\R)$ such that 
\[
A(x) = P(f(x))P(x)^{-1}, \; \; \text{for $\mu$-a.e. $x \in \Sigma$.}
\]
Then $P$ coincides $\mu$-a.e. with an $\alpha$-H\"older continuous function $\hat{P}$ satisfying the same equation for every $x \in \Sigma$.  
\end{thm}

The equation $A(x) = P(f(x))P(x)^{-1}$ is referred to as the \emph{cohomological equation} over $f$;  Theorem \ref{coboundary} states that, under the hypothesis that $A$ is H\"older continuous, any measurable solution of this equation coincides $\mu$-a.e. with a H\"older continuous solution. 

We pause to give some background on previous work relating to Theorem \ref{coboundary}. There are two principal questions regarding the cohomology of linear cocycles which are still not completely understood. For both questions below we take $A$ and $B$ to be $\alpha$-H\"older continuous for some $\alpha > 0$. 

\begin{enumerate}
\item Suppose that $A$ and $B$ have conjugate periodic data in the sense that there is a continuous map $C: \Sigma \rightarrow GL(d,\R)$ such that $A^{m}(p) = C(p)B^{m}(p)C(p)^{-1}$ for every periodic point $p$ of $f$ of period $m$. Are $A$ and $B$ continuously cohomologous?

\item Suppose that $A$ and $B$ are measurably cohomologous. Are they then continuously cohomologous?
\end{enumerate}

The answers to these two questions are ``yes" in the case where $d = 1$, i.e., when $A$ and $B$ can be viewed as additive real-valued cocycles. This is due to the pioneering work of Livsic \cite{Liv}. When $A$ and $B$ take values in a non-abelian group, the answers to these questions become considerably more subtle. In the case of Question (1) when $B(x)  =  Id_{\R^{d}}$ for every $x \in \Sigma$ a positive answer was obtained by Kalinin \cite{Kal}; we refer to that paper for a survey of previous work on that subject. For more general $B$ a positive answer to Question (1) has also been obtained simultaneously by Backes \cite{B15} and Sadovskaya \cite{S15} under the additional assumption that $B$ is \emph{fiber bunched} (see Definition \ref{fiber} below). The fiber bunching condition holds in an open neighborhood of the identity cocycle. Question (1) remains open in its full generality. 

As for Question (2), Sadovskaya showed in the case that $A$ is fiber bunched and $B \equiv  Id_{\R^{d}}$ that if $A$ and $B$ are measurably cohomologous then they are continuously cohomologous \cite{S15}. Our Theorem \ref{coboundary} improves on her result by removing the fiber bunching condition on $A$. There are examples of H\"older continuous cocycles $A$ and $B$ arbitrarily close to the identity which are measurably cohomologous but not continuously cohomologous that have been constructed by Pollicott and Walkden \cite{PW01} (see also \cite{S15}). It's thus unclear whether a general theorem is possible in the case of Question (2). 

We now introduce the main theorem of this paper from which both Theorems \ref{hyperbolic rigidity} and \ref{coboundary} will be deduced.  A \emph{conformal structure} on $\R^{d}$ is an equivalence class of inner products on $\R^{d}$, where two inner products $\langle , \rangle$ and $(,)$ are considered equivalent if there is a positive real number $\la > 0$ such that $\langle v,w \rangle = \la (v,w)$ for every $v,w \in \R^{d}$. We denote the space of conformal structures on $\R^{d}$ by $\mathcal{C}^{d}$. The space $\mathcal{C}^{d}$ can be naturally identified with the Riemannian symmetric space $SL(d,\R)/SO(d,\R)$ of real positive definite symmetric matrices of determinant 1 which carries a left-invariant nonpositively curved Riemannian metric $g$. $GL(d,\R)$ acts transitively on $\mathcal{C}^{d}$ via the action 
\[
B_{*}\eta = \det(B B^{T})^{-\frac{2}{d}} B^{T} \eta B
\]
for $B \in GL(d,\R)$ and a positive definite matrix $\eta \in SL(d,\R)/SO(d,\R)$, where $B^{T}$ denotes the transpose of $B$. This action is an isometry with respect to the metric $g$. For more details and proofs related to these facts, see \cite{KS10}. For a matrix $B \in GL(d,\R)$ and a positive definite matrix $\eta \in  SL(d,\R)/SO(d,\R)$ we define 
\[
B[\eta] = B^{-1}_{*}\eta
\]
We let $\langle,\rangle$ denote the standard Euclidean inner product on $\R^{d}$ and $\|\cdot \|$ the associated norm. 

\begin{defn}\label{defn: preserve conformal}
Let $\mathcal{A}$ be a linear cocycle over $f$. We say that $\mathcal{A}$ preserves a measurable conformal structure if there is a measurable map $\eta : \Sigma \rightarrow \mathcal{C}^{d}$ defined $\mu$-a.e. such that 
\[
A(x)[\eta_{x}] = \eta_{f(x)},
\]
for $\mu$-a.e. $x \in \Sigma$. We say that $\mathcal{A}$ preserves a continuous conformal structure if $\eta$ is continuous and the above equality holds for every $x \in \Sigma$. 
\end{defn}

We can now state the main theorem of this paper, 

\begin{thm}\label{conformalrigidity}
Let $\mathcal{A}$ be an $\alpha$-H\"older continuous linear cocycle over $f$. If $\mathcal{A}$ preserves a measurable conformal structure $\eta: \Sigma \rightarrow \mathcal{C}^{d}$ then $\eta$ coincides $\mu$-a.e. with an $\alpha$-H\"older continuous conformal structure $\hat{\eta}: \Sigma \rightarrow \mathcal{C}^{d}$ which is invariant under $\mathcal{A}$. 
\end{thm}

To discuss previous results related to this theorem we introduce the notion of fiber bunching for a linear cocycle, 

\begin{defn}\label{fiber}
A linear cocycle $\mathcal{A}$ over $f: \Sigma \rightarrow \Sigma$ is \emph{fiber bunched} if, for some choice of $\tau > 0$, $\mathcal{A}$ is $\alpha$-H\"older continuous with respect to the metric $\rho_{\tau}$ (as defined in \eqref{metric def}) and there are constants $0 < \xi < 1$ and $L > 0$ such that 
\[
\|A^{n}(x)\|\|A^{n}(x)^{-1}\| e^{-|n|\alpha \tau} < L\xi^{|n|}
\]
for every $n \in \Z$.
\end{defn} 
The quantity $\|A^{n}(x)\|\|A^{n}(x)^{-1}\|$ measures the degree to which $A^{n}$ fails to be conformal with respect to the Euclidean inner product on $\R^{d}$. The fiber bunching assumption implies that the growth rate of this failure of conformality is uniformly dominated by the expansion/contraction rates of the base system. 

Fiber bunching is a common assumption among recent theorems proven about linear cocycles over hyperbolic systems: these include continuity of Lyapunov exponents \cite{BBB}, \cite{BV}, simplicity of Lyapunov exponents \cite{BV04}, cohomology of H\"older continuous cocycles \cite{B15},  \cite{S15}, and characterizations of linear cocycles with vanishing Lyapunov exponents \cite{AV10}. There are examples due to Bocker-Viana and the author \cite{BV}, \cite{Bu2}, illustrating the necessity of the fiber bunching condition for unconditional statements regarding continuity of the Lyapunov exponents, but beyond these examples very little is known about the necessity of the fiber bunching hypothesis for these theorems. 

Without the fiber bunching hypothesis the conclusion of Theorem \ref{conformalrigidity} was also previously known to hold under additional boundedness hypotheses on the $\mathcal{A}$-invariant conformal structure $\eta$. The best previous result in this direction is due to Sadovskaya and de la Llave \cite{LS}, who showed that if the function $x \rightarrow g(\eta_{x}, \langle,\rangle)$ is in $L^{p}(\mu)$ for $p$ sufficiently large then $\eta$ coincides $\mu$-a.e. with a continuous $\mathcal{A}$-invariant conformal structure. 

We remark that all of the results of this paper apply equally well to H\"older continuous cocycles over transitive Anosov diffeomorphisms or Anosov flows where the base measure $\mu$ is taken to be the equilibrium state of a H\"older continuous potential. For transitive Anosov diffeomorphism the reduction to a H\"older continuous cocycle over a subshift of finite type via a Markov partition follows standard techniques. For transitive Anosov flows the reductions to the results of this paper are less trivial but still straightforward. The reductions in that case can be done via the techniques used in the proof of Theorem \ref{hyperbolic rigidity} below. 

 We thank Amie Wilkinson for many helpful discussions regarding the contents of this paper. We also thank the anonymous referee for numerous useful suggestions which greatly improved the quality of the paper. 



\section{Reduction to Theorem \ref{conformalrigidity}}\label{reduction}
In this section we will show how Theorem \ref{conformalrigidity} can be used to obtain some corollaries of independent interest which we can use to prove Theorems \ref{hyperbolic rigidity} and \ref{coboundary}. We begin by using Theorem \ref{conformalrigidity} to remove the fiber bunching hypothesis from one of the main theorems of a recent work of Sadovskaya\cite{S15}. 

Following Sadovskaya, a continuous linear cocycle $\mathcal{A}$ with generator $A$ over $f$ is \emph{uniformly quasiconformal} if there is a constant $L > 0$ such that for every $x \in \Sigma$, 
\[
\|A^{n}(x)\| \|A^{n}(x)^{-1}\| \leq L \;\; \text{for every $n \in \Z$.}.
\]
If $\mathcal{A}$ preserves a continuous invariant conformal structure $\eta$ then $\mathcal{A}$ is uniformly quasiconformal: since $\eta$ is continuous and $\Sigma$ is compact there is a constant $C$ such that for all $v \in \R^{d}$ and $x \in \Sigma$, 
\[
C^{-1}\|v\| \leq \sqrt{\eta_{x}(v,v)} \leq C\|v\|.
\]
Since $\eta$ is invariant there is a continuous function $\psi: \Sigma \rightarrow \R$ such that for every $v \in \R^{d}$
\[
\sqrt{\eta_{x}(A(x)v,A(x)v)} = \psi(x)\sqrt{\eta_{x}(v,v)}.
\]
Set $\psi^{n}(x) = \prod_{i=0}^{n-1} \psi(f^{i}(x))$. Putting these inequalities together we compute that 
\[
\|A^{n}(x)v\| \|A^{n}(x)^{-1}v\| \leq C^{2} \psi^{n}(x) \cdot (\psi^{n}(x))^{-1} \cdot \eta_{x}(v,v) \leq C^{4}\|v\|^{2}
\]
Hence $\mathcal{A}$ is uniformly quasiconformal with constant $L = C^{4}$. The reverse implication that a uniformly quasiconformal cocycle preserves a continuous conformal structure also holds when $\mathcal{A}$ is H\"older continuous \cite{KS10}. 

\begin{cor}\label{Livsicmeasurable}
Let $\mathcal{A}$ and $\mathcal{B}$ be $\alpha$-H\"older continuous linear cocycles over $f$ such that $\mathcal{B}$ preserves a measurable conformal structure. Suppose that there is a measurable function $P : \Sigma \rightarrow GL(d,\R)$ such that 
\[
A(x) = P(f(x))B(x)P(x)^{-1}, \; \; \text{for $\mu$-a.e. $x \in \Sigma$.}
\]
Then  both $\mathcal{A}$ and $\mathcal{B}$ preserve an $\alpha$-H\"older continuous conformal structure and $P$ coincides $\mu$-a.e. with an $\alpha$-H\"older continuous function. 
\end{cor}

\begin{proof}
By assumption $\mathcal{B}$ preserves a measurable conformal structure $\eta$. The cocycle $\mathcal{A}$ then preserves the measurable conformal structure $P[\eta]$. By Theorem \ref{conformalrigidity}, both $\eta$ and $P[\eta]$ coincide $\mu$-a.e. with a $\alpha$-H\"older continuous conformal structures $\widehat{\eta}$ and $\widehat{P[\eta]}$ which are $\mathcal{B}$-invariant and $\mathcal{A}$-invariant respectively. It follows from the remarks above that both $\mathcal{A}$ and $\mathcal{B}$ are uniformly quasiconformal and thus satisfy the hypotheses of Sadovskaya's theorem \cite[Theorem 2.7]{S15} from which we conclude that $P$ coincides $\mu$-a.e. with an $\alpha$-H\"older continuous function.  
\end{proof}

Theorem \ref{coboundary} follows from Corollary \ref{Livsicmeasurable} by taking $B \equiv Id_{\R^{d}}$. 

Before proving Theorem \ref{hyperbolic rigidity}, we obtain a corollary of Theorem \ref{conformalrigidity} for linear cocycles over a subshift of finite type which will be needed in the proof. We say that $\mathcal{A}$ is \emph{measurably reducible} (with respect to the measure $\mu$) if there is some $1 \leq r \leq d-1$ and some $n \geq 1$ such that there is a measurable family of $r$-dimensional linear subspaces $E_{x} \subset \R^{d}$ parametrized by $x \in \Sigma$ such that $A^{n}(x)(E_{x}) = E_{f^{n}(x)}$ for $\mu$-a.e. $x \in \Sigma$. We say that $\mathcal{A}$ is \emph{measurably irreducible} if it is not measurably reducible. Our next corollary gives a measure-theoretic criterion for a H\"older continuous cocycle to preserve a continuous conformal structure.

\begin{cor}\label{irreducible}
Let $\mathcal{A}$ be an $\alpha$-H\"older linear cocycle over $f$. Suppose that $\mathcal{A}$ is measurably irreducible. Then $\mathcal{A}$ preserves an $\alpha$-H\"older continuous conformal structure. 
\end{cor}

\begin{proof}
By Zimmer's amenable reduction theorem \cite{Zim} there is a measurable map $P: \Sigma \rightarrow GL(d,\R)$ defined $\mu$-a.e. such that $B(x) := P(f(x))^{-1} A(x) P(x)$ takes values in a subgroup $G$ of $GL(d,\R)$ which contains an amenable subgroup $H$ of finite index. There is then an $n \geq 1$ such that $B^{n}(x) \in H$ for $\mu$-a.e. $x \in \Sigma$. 

According to the classification of maximal amenable subgroups of $GL(d,\R)$, either $H$ stabilizes a proper subspace $\{0\} \subsetneq V \subsetneq \R^{d}$ or $H$ is a subgroup of the conformal linear automorphisms $\R^{\times} \times SO(d,\R)$ (see \cite{Moo} for this classification and \cite[Theorem 3.4]{KS} for a more detailed discussion of this implication). In the first case we conclude that $A^{n}$ leaves $\mu$-a.e. invariant the measurable family of subspaces $E_{x} = P(x)(V)$ and thus $A$ is measurably reducible. In the second case $A$ preserves the measurable conformal structure induced by $P[\langle,\rangle]$, where we recall that $\langle, \rangle$ denotes the standard Euclidean inner product on $\R^{d}$. By Theorem \ref{conformalrigidity} $P[\langle,\rangle]$ coincides $\mu$-a.e. with an $\alpha$-H\"older continuous conformal structure which is preserved by $\mathcal{A}$. 
\end{proof}

We lastly show how to obtain Theorem \ref{hyperbolic rigidity} from Corollary \ref{irreducible}. We let $T^{1}M$ denote the unit tangent bundle of a closed negatively curved manifold $M$ with $\dim M \geq 3$ and let $g^{t}$ be the geodesic flow on $T^{1}M$. Since $M$ is negatively curved, the geodesic flow is an Anosov flow and thus there is a $Dg^{t}$-invariant splitting $T(T^{1}M) = E^{u} \oplus E^{c} \oplus E^{s}$ where $E^{u}$ is exponentially expanded by $Dg^{t}$, $E^{s}$ is exponentially contracted by $Dg^{t}$, and $E^{c}$ is tangent to the flow direction of $g^{t}$. We also have $\dim E^{u} = \dim M - 1 \geq 2$ and we know that $E^{u}$ is an $\alpha$-H\"older continuous subbundle of $T(T^{1}M)$ for some $\alpha > 0$. These assertions are all well-known, see \cite{HK} for instance for proofs. 

Let $\varphi: T^{1}M \rightarrow \R$ be a H\"older continuous potential and $\mu_{\varphi}$ the $g^{t}$-invariant probability measure which is the equilibrium state associated to $\varphi$.  Equilibrium states include the Liouville volume on $T^{1}M$ and the Bowen-Margulis measure of maximal entropy; we refer to \cite{BR} for details. 

Lastly for a subshift of finite type $\Sigma$ and a continuous function $\psi: \Sigma \rightarrow (0,\infty)$ we let $\Sigma_{\psi}$ be the space
\[
\Sigma_{\psi} = \{(x,t) \in \Sigma \times [0,\infty]: t \leq \psi(x) \}/(x,\psi(x)) \sim (f(x),0)
\]
and let $F^{t}$ be the special flow on $\Sigma_{\psi}$ defined by $F^{s}(x,t) = (x,t+s)$ (up to the equivalence relation above).  

\begin{proof}[Proof of Theorem \ref{hyperbolic rigidity}]
Fix a Riemannian metric on $T^{1}M$ and let $| \cdot |$ be the norm on $E^{u}$ given by restricting the Riemannian norm to this bundle. We claim that it suffices to prove that $Dg^{t}|_{E^{u}}$ is uniformly quasiconformal in the sense that there is a constant $C > 0$ such that for all $v \in E^{u}$ and $t \in \R$, 
\[
\left|Dg^{t}(v)\right| \cdot \left|(Dg^{t}(v))^{-1}\right| \leq C |v|^{2}.
\]
This immediately implies that for each periodic point $p$ of $g^{t}$ of period $\l(p)$ the complex eigenvalues of $Dg^{\l(p)}: E^{u}_{p} \rightarrow E^{u}_{p}$ are all equal in absolute value. Hence by work of the author \cite{Bu1}[Theorem 1.1] this implies that $M$ has constant negative curvature. 

For this proof we set $d = \dim E^{u}$. We will derive from the derivative cocycle $Dg^{t}|_{E^{u}}$ a H\"older continuous cocycle $\mathcal{A}$ over a subshift of finite type $f: \Sigma \rightarrow \Sigma$ together with an $f$-invariant measure $\nu$ corresponding to the $g^{t}$-invariant measure  $\mu_{\varphi}$ such that if $Dg^{t}|_{E^{u}}$ is measurably irreducible with respect to $\mu_{\varphi}$ then $\mathcal{A}$ is measurably irreducible with respect to $\nu$. We then apply Corollary \ref{irreducible} to $\mathcal{A}$ to derive a continuous invariant conformal structure. We will show that this implies that the cocycle $Dg^{t}|_{E^{u}}$ is uniformly quasiconformal which completes the proof by our work above. 

From the remarks above it suffices to show that $Dg^{t}_{E^{u}}$ preserves a continuous conformal structure. We first recall aspects of the construction of a Markov partition for an Anosov flow \cite{R73}. The geodesic flow $g^{t}$ is a topologically mixing Anosov flow and thus admits a Markov partition into parallelograms $\{P_{i}: 1 \leq i \leq k\}$, whose sides are composed of orbits of $g^{t}$ and leaves of the stable and unstable foliations for $g^{t}$. Given any $\e > 0$ these parallelograms can be chosen such that the bases $\{Q_{i} :1 \leq i \leq k\}$ transverse to the orbits of $g^{t}$ satisfy $\text{diam}\, Q_{i} < \e$. We choose $\e$ small enough that the bundle $E^{u}$ is trivial on each base $Q_{i}$. 

There is then a subshift of finite type $f: \Sigma \rightarrow \Sigma$ on $k$ symbols, a H\"older continuous roof function $\psi: \Sigma \rightarrow (0,\infty)$, and a H\"older continuous surjection $h: \Sigma_{\psi} \rightarrow T^{1}M$ such that the flow $F^{t}$ on $\Sigma_{\psi}$ satisfies $h \circ F^{t} = g^{t} \circ h$. Furthermore, identifying $\Sigma$ with $\Sigma \times \{0\} \subset \Sigma_{\psi}$,  each cylinder $[0;i]$, $1 \leq i \leq k$, satisfies $h([0;i]) = Q_{i}$. 

There is an ergodic, fully supported measure $\hat{\nu}$ on $\Sigma_{\psi}$ with local product structure such that $h_{*}\hat{\nu} = \mu_{\varphi}$ and such that $h$ is a bijection from a full measure subset $\Omega$ of $\Sigma_{\psi}$ to a full measure subset $h(\Omega)$ of $T^{1}M$. We let $\nu$ be the probability measure on $\Sigma$ obtained from pushing forward the measure $\hat{\nu}$ by the natural projection $\Sigma_{\psi} \rightarrow \Sigma$. The measure $\nu$ is also ergodic, fully supported, and has local product structure\cite{BR}.

The bundle $E^{u}$ pulls back to a H\"older vector bundle $\hat{\E}$ over $\Sigma_{\psi}$. We let $\E$ be the bundle over $\Sigma$ obtained from pulling back $\hat{\E}$ by the identification $\Sigma \rightarrow \Sigma \times \{0\} \subset \Sigma_{\psi}$. Since $E^{u}$ is a H\"older continuous bundle which is trivial over each set $Q_{i}$ there is a H\"older trivialization $L: \Sigma \times \R^{d} \rightarrow \E$.  The restriction $Dg^{t}|E^{u}$ of the derivative cocycle $Dg^{t}$ of $g^{t}$ to the unstable bundle $E^{u}$ is H\"older (from the H\"older continuity of $E^{u}$). We let  $\hat{A}^{t}$ denote the pullback of the linear cocycle $Dg^{t}|E^{u}$ to a linear cocycle on $\hat{\mathcal{E}}$ over $F^{t}$. This induces a H\"older cocycle $\mathcal{A}$ with generator  $A: \Sigma \rightarrow GL(d,\R)$ over $f: \Sigma \rightarrow \Sigma$ by taking the time $\psi(x)$ map $\hat{A}^{\psi}(x)$ over $x \in \Sigma \times \{0\}$ and then using the trivialization $L$ of $\E$. 

We claim that $\mathcal{A}$ is measurably irreducible. Suppose that there is a proper measurable invariant subbundle $V$ for $\mathcal{A}$. For each $x\in \Sigma$ at which $V$ is defined we let $V_{x} \subset \R^{d}$ be the subspace for this bundle at $x$. For each $x \in \Sigma$ at which $V_{x}$ is defined and each $t \in \R$ we define $\hat{V}_{(x,t)} = \hat{A}^{t}(L(V_{x}))$. From the local decomposition of $\hat{\nu}$ into a product of $\nu$ and Lebesgue measure on orbits together with the $\nu$-a.e. $\mathcal{A}$-invariance of $V$ it follows that $\hat{V}$ is a $\hat{\nu}$-a.e. well-defined measurable subbundle of $\hat{\E}$ which is also $\hat{A}^{t}$-invariant $\hat{\nu}$-a.e. From the construction of $\hat{\E}$ this implies that there is a proper measurable subbundle of $E^{u}$ which is defined and $Dg^{t}$-invariant $\mu_{\varphi}$-a.e. This contradicts the hypothesis that $Dg^{t}|_{E^{u}}$ is measurably irreducible. 

Having shown $\mathcal{A}$ is measurably irreducible, we apply Corollary \ref{irreducible} to conclude that $\mathcal{A}$ preserves a continuous invariant conformal structure $\eta$. From this we construct a continuous invariant conformal structure $\hat{\eta}$ for $\hat{A}^{t}$ by the formula $\hat{\eta}_{(x,t)} = (\hat{A}^{t} \circ L)_{*}[\eta_{x}]$ for each $x \in \Sigma$ and $t \in \R$. 

Recall from the construction of a Markov partition that there is a full $\hat{\nu}$-measure subset $\Omega \subset \Sigma_{\psi}$ on which $h : \Omega \rightarrow h(\Omega)$ is a continuous bijection onto a full $\mu_{\varphi}$-measure subset $h(\Omega)$ of $T^{1}M$. Let $S: \hat{\E}|_{\Omega} \rightarrow E^{u}|_{h(\Omega)}$ be the continuous linear identification on fibers coming from the construction of $\hat{\E}$ as the pullback of $E^{u}$. 

Since $\hat{\eta}$ is continuous on $\Sigma_{\psi}$, we conclude that there is a constant $\gamma > 0$ such that for any $v \in E^{u}|_{h(\Omega)}$ we have 
\[
 \gamma^{-1} \cdot \hat{\eta}(S^{-1}(v), S^{-1}(v)) \leq  |v|^{2} \leq \gamma \cdot \hat{\eta}(S^{-1}(v), S^{-1}(v)).
\]
Since $\hat{\eta}$ is invariant under $\hat{A}^{t}$ and $S \circ \hat{A}^{t} = Dg^{t} \circ S$, we conclude using the above inequality that for any $v \in E^{u}|_{h(\Omega)}$ and $t \in \R$ we have 
\begin{align*}
\left|Dg^{t}(v)\right|^{2} &\leq \gamma \cdot \hat{\eta}(S^{-1}(Dg^{t}(v)), S^{-1}(Dg^{t}(v))) \\
& = \gamma \cdot \hat{\eta}(S^{-1}(v), S^{-1}(v)) \\
& \leq \gamma ^{2} |v|^{2}
\end{align*}
Since $(Dg^{t}(v))^{-1} = Dg^{-t}(g^{t}(v)$, this readily implies that 
\[
\left|Dg^{t}(v)\right| \cdot \left|(Dg^{t}(v))^{-1}\right| \leq \gamma^{2} |v|^{2}
\] 
for all $v \in E^{u}|_{h(\Omega)}$ and $t \in \R$. Since $h(\Omega)$ is a dense subset of $T^{1}M$ we conclude that this inequality actually holds for all $v \in E^{u}$ and $t \in \R$. We conclude that $Dg^{t}|_{E^{u}}$ is a uniformly quasiconformal linear cocycle over $g^{t}$, as desired.


\end{proof}


\section{Definitions and Reductions}\label{shift definitions}
In this section we define some common objects associated to $\Sigma$ and make some preliminary reductions before beginning the proof of Theorem \ref{conformalrigidity}. 

\subsection{Definitions} We define
\[
\Sigma^{u} = \{(x_{n})_{n \geq 0}: q_{x_{n}x_{n+1}} = 1 \;\text{for all $n \geq 0$}\},
\]
\[
\Sigma^{s} = \{(x_{n})_{n < 0}: q_{x_{n}x_{n+1}} = 1 \;\text{for all $n \leq -2$}\},
\] 
to be the sets of one-sided right and left infinite sequences respectively which are associated to $Q$. We have projections $\pi_{u}: \Sigma \rightarrow \Sigma^{u}$ and $\pi_{s}: \Sigma \rightarrow \Sigma^{s}$ obtained by dropping all of the negative coordinates and all of the nonnegative coordinates respectively of a sequence in $\Sigma$. We let $f_{u}$ denote the left shift on $\Sigma^{u}$ and $f_{s}$ denote the right shift on $\Sigma^{s}$. 

We define the \emph{local stable set} of $x \in \Sigma$ to be 
\[
W^{s}_{loc}(x) = \{(y_{n})_{n \in\mathbb{Z}} \in \Sigma:x_{n}  = y_{n} \; \text{for all $n \geq 0$}\},
\] 
and the \emph{local unstable set} to be 
\[
W^{u}_{loc}(x) = \{(y_{n})_{n \in \mathbb{Z}} \in \Sigma :x_{n}  = y_{n} \; \text{for all $n \leq 0$}\}.
\]
For $m \in \mathbb{Z}$ and $a_{0},\dots,a_{k} \in \N$, we define the cylinder notation
\[
[m;a_{0},\dots,a_{k}] = \{x \in \Sigma: x_{m+i} = a_{m+i}, \, 0 \leq i \leq k\}.
\]
For $m \geq 0$, we let $[m;a_{0},\dots,a_{k}]^{u}$ denote the corresponding cylinder in $\Sigma^{u}$ and for $m \leq -k$,  $[m;a_{0},\dots,a_{k}]^{s}$ is the corresponding cylinder in $\Sigma^{s}$.

For two points $x,y \in [0;i]$, $1 \leq i \leq \l$, we define $[x,y]$ to be the unique point in the intersection $W^{u}_{loc}(x) \cap W^{s}_{loc}(y)$. 

We let $\mu$ be a fully supported ergodic $f$-invariant probability measure on $\Sigma$. We let $\mu^{u} = (\pi_{u})_{*}\mu$ and $\mu^{s} = (\pi_{s})_{*}\mu$ denote the projections of $\mu$ to $\Sigma^{u}$ and $\Sigma^{s}$ respectively. For a subset $E$ of $\Sigma$ we let $\mu | E$ denote the restriction of $\mu$ to $E$. Note that for each $1 \leq i \leq \l$ we have a natural homeomorphism $[0;i] \rightarrow \pi_{s}([0;i])  \times \pi_{u}([0;i]))$ given by $\pi_{s} \times \pi_{u}$. We say that $\mu$ has \emph{local product structure} if there is a continuous function $\xi: \Sigma \rightarrow (0,\infty)$ such that 
\[
\mu | [0;i] = \xi \cdot  (\mu^{s}|\pi_{s}([0;i]) \times \mu^{u}|\pi_{u}([0;i])).
\]
We will always assume that $\mu$ has local product structure in this paper. 

As a consequence of the local product structure of $\mu$, the measures $\mu^{u}$ and $\mu^{s}$ admit Jacobians with respect to the dynamics $f_{u}$ and $f_{s}$ \cite[Lemma 2.2]{BV04},

\begin{prop}\label{jacobians}
There are continuous functions $J_{u}: \Sigma^{u} \rightarrow (0,\infty)$ and $J_{s}: \Sigma^{s} \rightarrow (0,\infty)$ such that for each $1 \leq i,j \leq \l$ and each Borel subset $K_{u} \subseteq [0;i]^{u}$, 
\[
\mu^{u}(K_{u}) = \int_{f_{u}^{-1}(K_{u}) \cap [0;j,i]^{u}}J_{u}\, d\mu^{u},
\]
and likewise for each Borel subset $K_{s} \subseteq [0;i]^{s}$, 
\[
\mu^{s}(K_{s}) = \int_{f_{s}^{-1}(K_{s}) \cap [-1;i,j]^{s}}J_{s} \, d\mu^{s},
\]
In particular, if $\mu^{u}(K_{u}) > 0$ then $\mu^{u}(f_{u}^{-1}(K_{u})) \cap [0;j,i]^{u}) > 0$ for each $1 \leq j \leq \l$, and the analogous statement is true for $\mu^{s}$ and $K_{s}$. 
\end{prop}

Let $\nu$ be any $f$-invariant ergodic probability measure on $\Sigma$. The \emph{extremal Lyapunov exponents} of $\mathcal{A}$ with respect to $\nu$ are defined to be 
\[
\la_{+}(\mathcal{A},\nu) = \inf_{n \geq 1} \frac{1}{n}\int_{\Sigma} \log \|A^{n}\| \,d\nu,
\]
\[
\la_{-}(\mathcal{A},\nu) = \sup_{n \geq 1} \frac{1}{n}\int_{\Sigma} \log \|A^{-n}\|^{-1} \,d\nu
\]
By Kingman's subadditive ergodic theorem \cite{K68}, for $\nu$-a.e. $x \in \Sigma$ we have 
\[
\lim_{n \rightarrow \infty} \frac{1}{n}\log \|A^{n}(x)\| =  \la_{+}(\mathcal{A},\nu),
\]
and the analogous statement is true for $\la_{-}(\mathcal{A},\nu)$. We will be particularly interested in the extremal Lyapunov exponents of probability measures supported on periodic orbits of $f$. 

\subsection{Reductions}\label{subsec: reduction} We make some reductions before beginning the proof of Theorem \ref{conformalrigidity}.  Observe that a function is $\alpha$-H\"older continuous with respect to the distance $\rho_{\tau}$ if and only if it is Lipschitz continuous with respect to $\rho_{\alpha\tau}$. Hence it suffices to prove Theorem \ref{conformalrigidity} in the case that $A$ is Lipschitz continuous with respect to $\rho_{\tau}$. From now on we fix $\tau \in (0,\infty)$ and define $\rho: = \rho_{\tau}$. 

We next reduce to the case where $A$ takes values in $SL(d,\R)$. If we define 
\[ 
B = \det(A)^{-\frac{1}{d}} \cdot A,
\]
then $\det(B(x)) = 1$ for every $x \in \Sigma$ so $B$ takes values in $SL(d,\R)$. It is clear that $B$ is Lipschitz with respect to $\rho$ if and only if $A$ is Lipschitz with respect to $\rho$. Let $\mathcal{B}$ be the linear cocycle generated by $B$. If $\eta: \Sigma \rightarrow \mathcal{C}^{d}$ defines a measurable conformal structure on $\Sigma \times \R^{d}$, then it is clearly $\mu$-a.e. $\mathcal{A}$-invariant if and only if it is $\mu$-a.e. $\mathcal{B}$-invariant, since these cocycles are conformally equivalent to each other. Thus for the rest of the paper we will assume that $\mathcal{A}$ takes values in $SL(d,\R)$. As a consequence, for $\mu$-a.e. $x \in \Sigma$ we have the equality 
\[
\eta_{f(x)}(A(x)v,A(x)w) = \eta_{x}(v,w), \; \text{for every $v,w \in \R^{d}$}.
\]
Lastly we reduce to the case where $f: \Sigma \rightarrow \Sigma$ is topologically mixing. By the spectral decomposition theorem for hyperbolic basic sets there is an integer $k \geq 1$ such that we can write $\Sigma = \bigsqcup_{i = 1}^{k} \Sigma_{i}$ for closed subsets $\Sigma_{i}$ of $\Sigma$ such that $f(\Sigma_{i}) = \Sigma_{i+1}$, where the index is taken mod $k$ and $f^{k}|\Sigma_{i}$ is a topologically mixing subshift of finite type for $1 \leq i \leq k$. The normalized restriction $\mu_{i}$ of $\mu$ to $\Sigma_{i}$ is an $f^{k}$-invariant ergodic, fully supported probability measure with local product structure and $A^{k}|\Sigma_{i}$ generates a Lipschitz continuous linear cocycle $\mathcal{A}_{i}^{k}$ over $f^{k}$. 

If $\mathcal{A}$ preserves a measurable conformal structure $\eta$, then $\mathcal{A}_{i}^{k}$ preserves the measurable conformal structure $\eta | [0;i]$. Hence if Theorem \ref{conformalrigidity} holds for topologically mixing subshifts of finite type, then $\eta | [0;i]$ coincides $\mu_{i}$-a.e. with a Lipschitz continuous invariant conformal structure for $\mathcal{A}^{k}_{i}$. This holds for each $1 \leq i \leq k$; we thus conclude that $\eta$ agrees $\mu$-a.e. with a Lipschitz continuous conformal structure on $\Sigma \times \R^{d}$ which is necessarily $\mathcal{A}$-invariant.

\section{Proof of Main Theorem}\label{sec: proof}
We let $f: \Sigma \rightarrow \Sigma$ be a topologically mixing subshift of finite type equipped with a metric $\rho = \rho_{\tau}$ as defined in \eqref{metric def}. We let $\mathcal{A}$ be a Lipschitz continuous $SL(d,\R)$-valued cocycle over $f$. Let $\eta: \Sigma \rightarrow SL(d,\R)/SO(d,\R)$ be an $\mathcal{A}$-invariant measurable family of inner products on $\R^{d}$ for which the equality $A(x)[\eta_{x}] = \eta_{f(x)}$ holds for $\mu$-a.e. $x$. In this sectionw we will prove that there is a Lipschitz continuous $\mathcal{A}$-invariant inner product $\hat{\eta}$ which coincides $\mu$-a.e. with $\eta$. By the reductions of Section \ref{subsec: reduction} this completes the proof of Theorem \ref{conformalrigidity}.

For a periodic point $x$ of period $k$ we define $\nu_{x}$ to be the $f$-invariant probability measure supported on the orbit of $p$ which is given by the formula
\[
\nu_{x} = \frac{1}{k}\sum_{j = 0}^{k-1} \delta_{f^{j}(x)},
\] 
where $\delta_{p}$ denotes the Dirac probability measure supported at a point $p$. We use the shorthand $\la_{+}(p) = \la_{+}(\mathcal{A},\nu_{p})$ and $\la_{-}(p) =\la_{-} (\mathcal{A},\nu_{p})$ for the extremal Lyapunov exponents of $\mathcal{A}$ with respect to $\nu_{p}$. The key proposition we will prove in this section is the following, 

\begin{prop}\label{periodiczero}
Suppose that $\mathcal{A}$ preserves a measurable conformal structure. Then $\la_{+}(p) = 0$ for every periodic point $p$ of $f$. 
\end{prop}

In Section \ref{conclusion} at the end of the paper we apply work of Kalinin and Sadovskaya to show how Proposition \ref{periodiczero} implies Theorem \ref{conformalrigidity}.

The plan of the proof is the following: We let $\mathcal{D}(N,\theta)$ be the set of points $x \in \Sigma$ which satisfy the two conditions 
\[
\prod_{j = 0}^{s-1} \|A^{N}(f^{jN}(x))\| \|A^{N}(f^{jN}(x))^{-1}\| \leq e^{sN\theta} \; \; \text{for all $s \geq 1$},
\]
and 
\[
\prod_{j = 0}^{s-1} \|A^{-N}(f^{-jN}(x))\| \|A^{-N}(f^{-jN}(x))^{-1}\| \leq e^{sN\theta} \; \; \text{for all $s \geq 1$}.
\]
The cocycle $\mathcal{A}$ is fiber bunched if and only if there is a $\theta < \tau$ and $N \in \N$ such that $\mathcal{D}(N,\theta) = \Sigma$. In general the sets $\mathcal{D}(N,\theta)$ with $\theta < \tau$ isolate compact subsets of the space $\Sigma$ on which we can treat $\mathcal{A}$ as if it were fiber bunched. 

A straightforward argument will show that if $\mathcal{A}$ preserves a measurable conformal structure $\eta$ then, given any $\theta > 0$, we have that for $\mu$-a.e. $x \in \Sigma$ there is an $N > 0$ such that $x \in \mathcal{D}(N,\theta)$. We show in Lemma \ref{smallperiodiczero} below that any periodic point $p$ satisfying $\la_{+}(p) - \la_{-}(p)  < \tau$ lies in $\mathcal{D}(N,\theta)$ for some $\theta < \tau$ and some $N$. We will show in Section \ref{extension} that $\eta$ admits a canonical $\mathcal{A}$-invariant (in an appropriate sense) extension $\hat{\eta}$ from a full measure subset of $\mathcal{D}(N,\theta)$ to the entire set. In particular, letting $k$ denote the period of $p$, we conclude that $A^{k}(p)$ preserves the inner product $\hat{\eta}_{p}$ on $\R^{d}$. This shows that if $\la_{+}(p)  - \la_{-}(p)< \tau$ then we actually have $\la_{+}(p) = 0$.

Thus there are no periodic points $p$ satisfying $0 < \la_{+}(p) - \la_{-}(p)< \tau$. The essential strategy to finish the proof is to show that if there exists a periodic point $x$ satisfying $\la_{+}(x) -\la_{+}(x) \geq \tau$ then by appropriately mixing the orbit of $x$ with the orbit of a periodic point $y$ satisfying $\la_{+}(y) = 0$ we can construct a new periodic point $p$ satisfying  $0 < \la_{+}(p) - \la_{-}(p)< \tau$. This would give a contradiction that shows that we must have $\la_{+}(p) = 0$ for all periodic points $p$. For technical reasons related to the difficulty of estimating Lyapunov exponents from below we will instead use a modification of this strategy below. 


\subsection{Holonomy invariance}\label{subsec: holonomies}
Our first observation is that the existence of a measurable invariant conformal structure implies that the extremal Lyapunov exponents of $\mathcal{A}$ with respect to $\mu$ must be zero,  

\begin{prop}\label{zeroexponents}
If there is a measurable function $\eta: \Sigma \rightarrow SL(d,\R)/SO(d,\R)$ such that $A(x)[\eta_{x}] = \eta_{f(x)}$ for $\mu$-a.e. $x$ then   $\la_{+}(\mathcal{A},\mu) = -\la_{-}(\mathcal{A},\mu) = 0$. 
\end{prop}

\begin{proof}
By Lusin's theorem we can find a compact subset  $K \subseteq \Sigma$ with positive $\mu$-measure on which $\eta$ is uniformly continuous. Let $\hat{f}: K \rightarrow K$ be the first return map to $K$ and $\hat{\mathcal{A}}$ the induced cocycle over $\hat{f}$ on $K$ with generator $\hat{A}$. By continuity of $\eta$ restricted to $K$, there is a constant $C \geq 1$ such that for all $x \in K$ and $v \in \R^{d}$ we have
\[
C^{-1}\|v\|^{2} \leq \eta_{x}(v,v) \leq C\|v\|^{2}
\]
For every $x \in K$, every $n \geq 1$, and each $v \in \R^{d}$,  we know that 
\[
C^{-1}\|\hat{A}^{n}(x)(v)\|^{2} \leq \eta_{\hat{f}^{n}(x)}(\hat{A}^{n}(x)(v),\hat{A}^{n}(x)(v)) \leq C\|\hat{A}^{n}(x)(v)\|^{2}
\]
Combining this with the fact that for every $x \in K$ we have $\hat{A}_{*}^{n}\eta_{x} = \eta_{\hat{f}^{n}(x)}$, we conclude that for every $x \in K$ we have 
\[
\|\hat{A}^{n}(x)(v)\|^{2} \leq C^{2} \|v\|^{2},
\]
for every $n \geq 1$. Hence $\la_{+}(\hat{\mathcal{A}},\mu_{K}) = 0$, where $\mu_{K} = \mu(K)^{-1}\mu | K$ is the induced invariant probability measure for $\hat{f}$ on $K$. 

By a direct application of Oseledet's multiplicative ergodic theorem (see \cite{V14} for instance) the extremal exponents of the induced cocycle $\hat{\mathcal{A}}$ are related to those of $\mathcal{A}$ by the equation $\la_{+}(\hat{\mathcal{A}},\mu_{K}) = \mu(K)^{-1}\la_{+}(\mathcal{A},\mu)$. This implies that $\la_{+}(\mathcal{A},\mu) = 0$. 
\end{proof}

 Since $\la_{+}(\mathcal{A},\mu) = 0$, we conclude from \cite[Corollary 2.4]{V08} that for any given $\theta > 0$ and $\mu$-a.e. $x \in \Sigma$, there is an $N > 0$ such that $x \in \mathcal{D}(N,\theta)$. When $\theta < \tau$ the cocycle $\mathcal{A}$ carries additional structure over $\mathcal{D}(N,\theta)$, 

\begin{prop}[{\cite[Proposition 2.5]{V08}}]\label{existenceholonomies}
Given $N,\theta$ with $\theta < \tau$, there exists $L = L(N,\theta) > 0$ such that for any $x \in \mathcal{D}(N,\theta)$ and any $y,z \in W^{s}_{loc}(x)$, 
\[
H^{s}_{yz} = \lim_{n \rightarrow \infty} A^{n}(z)^{-1}A^{n}(y),
\]
exists and satisfies $\|H^{s}_{yz}-I\| \leq L \cdot \rho (y,z)$ and $H^{s}_{yz} = H^{s}_{xz} \circ H^{s}_{yx}$. Similarly, if $y,z \in W^{u}_{loc}(x)$ then the limit
\[
H^{u}_{yz} = \lim_{n \rightarrow \infty} A^{n}(z)A^{n}(y)^{-1},
\]
exists and satisfies $\|H^{u}_{yz}-I\| \leq L \cdot \rho (y,z)$ and $H^{u}_{yz} = H^{u}_{xz} \circ H^{u}_{yx}$. 
\end{prop}

We refer to $H^{s}$ and $H^{u}$ as the stable and unstable holonomies of $\mathcal{A}$ respectively. We will show in this subsection that there is an $f$-invariant full measure subset $\Omega$ of $\Sigma$ on which $\eta$ is both $f$-invariant and holonomy invariant in the sense that a holonomy map between two points in $\Omega$ preserves $\eta$.  In the next lemma we show that $\eta$ is invariant under these stable and unstable holonomies on a full measure subset of $\Sigma$,

\begin{lem}\label{holonomyinvariance}
There is an $f$-invariant subset $\Omega$ of $\Sigma$, with $\mu(\Omega) = 1$, such that if $x,y \in \Omega$ and $y \in W^{*}_{loc}(x)$ then 
\[
H^{*}_{xy}[\eta_{x}] = \eta_{y},
\]
for $* \in \{s,u\}$, and furthermore for every $n \in \Z$ we have
\[
A^{n}(x)[\eta_{x}] = \eta_{f^{n}(x)},
\]
for every $x \in \Omega$. 
\end{lem}

\begin{proof}
We follow \cite[Proposition 4.4]{KS}, with minor adjustments to account for the fact that the holonomies $H^{s}$ and $H^{u}$ do not necessarily vary continuously on $\Sigma$. Fix $N > 0$ and $\theta < \tau$ such that $\mu(\mathcal{D}(N,\theta)) > \frac{9}{10}$. 

By Lusin's theorem we can find a compact subset $K \subset \Sigma$ with $\mu(K) > \frac{9}{10}$ on which $\eta$ is uniformly continuous and thus bounded. We note that 
\[
\frac{\mu(\mathcal{D}(N,\theta) \cap K)}{\mu(\mathcal{D}(N,\theta))} > \frac{4}{5}.
\]
Set $E = \mathcal{D}(N,\theta) \cap K$. Let $\hat{f}$ denote the first return map to $\mathcal{D}(N,\theta)$. Let $S \subset \mathcal{D}(N,\theta)$ be the subset of points of $\mathcal{D}(N,\theta)$ for which the frequency of visits to $E$ under iteration of $\hat{f}^{n}$ converges to $\frac{\mu(E)}{\mathcal{D}(N,\theta)} > \frac{1}{2}$ as $n \rightarrow \infty$. We have $\mu(S) = \mu(\mathcal{D}(N,\theta))$ by the Birkhoff ergodic theorem since $\hat{f}$ is ergodic. Hence for $\mu$-a.e. $x,y \in \mathcal{D}(N,\theta)$ with $y \in W^{s}_{loc}(x)$ there are infinitely many $n \geq 1$ such that both $\hat{f}^{n}(x)$ and $\hat{f}^{n}(y) \in E$. Fix any pair of such points $x,y$ with $y \in W^{s}_{loc}(x)$ and with the $A$-invariance property of $\eta$ for all iterates, and let $n_{i} \rightarrow \infty$ be a sequence with $x^{n_{i}}:=\hat{f}^{n_{i}}(x)$,  $y^{n_{i}}:= \hat{f}^{n_{i}}(y)\in E$ for each $i$. By $\hat{A}$-invariance of $\eta$ on the orbits of $x$ and $y$ and since pushing forward by $\hat{A}$ is an isometry with respect to the metric $g$, 
\begin{align*}
g(\eta(y), H^{s}_{xy}[\eta(x)]) &= g(\hat{A}^{n_{i}}(y)[\eta(y)], \hat{A}^{n_{i}}(y)H^{s}_{xy}[\eta(x)])\\
&= g(\eta(y^{n_{i}}), H^{s}_{x^{n_{i}}y^{n_{i}}}[\eta(x^{n_{i}})]) \\
&\leq g(\eta(y^{n_{i}}), \eta(x^{n_{i}})) + g(\eta(x^{n_{i}}), H^{s}_{x^{n_{i}}y^{n_{i}}}[\eta(x^{n_{i}})])
\end{align*}
Since $\eta$ is uniformly continuous on $E$, 
\[
g(\eta(x^{n_{i}}), \eta(y^{n_{i}})) \rightarrow 0,
\]
as $n_{i} \rightarrow \infty$. Also as $n_{i} \rightarrow \infty$, $H^{s}_{x^{n_{i}}y^{n_{i}}}$ converges uniformly to the identity map $I$ on $\R^{d}$. Since $\eta$ is bounded on $E$, this implies that 
\[
g(\eta(x^{n_{i}}), H^{s}_{x^{n_{i}}y^{n_{i}}}[\eta(x_{n_{i}})]) \rightarrow 0 \;\;\text{as $n_{i} \rightarrow \infty$}.
\]
We conclude that $\eta(y) = H^{s}_{xy}[\eta(x)]$, as desired. This holds for $\mu$-a.e. $x,y \in \mathcal{D}(N,\theta)$ with $y \in W^{s}_{loc}(x)$. Note that if $s \geq 1$ then $\mathcal{D}(N,\theta) \subset \mathcal{D}(sN,\theta)$. Taking the union over all $N > 0 $ of the sets $\mathcal{D}(N,\theta)$ for a fixed $\theta < \tau$ with $\mu(\mathcal{D}(N,\theta)) > \frac{9}{10}$ and also applying this argument replacing $H^{s}$ by $H^{u}$, we obtain a full $\mu$-measure subset $\Omega'$ of $\Sigma$ such that if $x,y \in \Omega$ and $y \in W^{*}_{loc}(x)$ then the holonomy $H^{*}_{xy}$ is defined according to Proposition \ref{existenceholonomies} and $\eta(y) = H^{*}_{xy}[\eta(x)]$ for $* \in \{s,u\}$. By intersecting $\Omega'$ with the full measure set on which the equality $A(x)[\eta_{x}] = \eta_{f(x)}$ holds we can arrange that this $A$-invariance property holds on $\Omega'$. We then set $\Omega = \cap_{n \in \Z} f^{n}(\Omega')$ to obtain the desired $f$-invariance property of $\Omega$. 
\end{proof}

\subsection{Extending $\eta$}\label{subsec: extension} For each $N > 0$ and $\theta < \tau$, the stable and unstable holonomies $H^{s}$ and $H^{u}$ exist and vary continuously on $\mathcal{D}(N,\theta)$. By construction $\eta$ is defined on $\Omega \cap \mathcal{D}(N,\theta)$ and invariant under the stable and unstable holonomies. By Proposition \ref{existenceholonomies} the restrictions of $H^{s}$ and $H^{u}$ to $\mathcal{D}(N,\theta)$ are Lipschitz with Lipschitz constant depending only on $N$ and $\theta$. Hence $\eta$ is uniformly continuous on $\Omega \cap \mathcal{D}(N,\theta)$. We would like to construct a continuous extension $\hat{\eta}$ of $\eta$ to $\mathcal{D}(N,\theta)$ which agrees $\mu$-a.e. with $\eta$, is invariant under $H^{s}$ and $H^{u}$ on all of $\mathcal{D}(N,\theta)$, and such that if $x, f^{n}(x) \in \mathcal{D}(N,\theta)$ then $A^{n}(x)[\hat{\eta}_{x}] = \hat{\eta}_{f^{n}(x)}$.

The natural choice for such an extension, since $\eta$ is uniformly continuous on $\Omega \cap \mathcal{D}(N,\theta)$, is to take the unique extension of $\eta$ to the closure $\overline{\Omega \cap \mathcal{D}(N,\theta)}$. However this may be a proper subset of $\mathcal{D}(N,\theta)$. The best we can guarantee is that 
 \[
 \text{supp}(\mu | \mathcal{D}(N,\theta)) \subseteq \overline{\Omega \cap \mathcal{D}(N,\theta)} \subseteq \mathcal{D}(N,\theta)
 \]
We overcome this issue using Lemma \ref{Lipschitzapprox} below which is one of the main steps in the proof of Proposition \ref{periodiczero}. This lemma makes critical use of the local product structure of $\mu$. 

\begin{lem}\label{Lipschitzapprox}
For each $N > 0$ and $\theta < \tau$ there is an $N_{*} \geq N$ and $ \theta \leq \theta_{*} < \tau$ such that 
\[
\mathcal{D}(N,\theta) \subseteq  \text{supp}(\mu | \mathcal{D}(N_{*},\theta_{*}))
\]
\end{lem}

\begin{proof}
Let $N > 0$, $\theta < \tau$ be given such that 
\[
\frac{1}{N}\int_{\Sigma}\log \|A^{N}(x)\| \cdot \|A^{N}(x)^{-1}\| \,d\mu < \theta,
\]
and set $\psi(x) = \frac{1}{N}\log \|A^{N}(x)\| \cdot \|A^{N}(x)^{-1}\|$. For any fixed $\theta > 0$ this inequality can always be arranged by taking $N$ large enough since $\la_{+}(\mathcal{A},\mu) = 0$. Set $S_{n}\psi(x) = \sum_{j=0}^{n-1}\psi(f(^{jN}(x)))$ and observe that if $x \in \mathcal{D}(N,\theta)$ then $\frac{S_{n}\psi(x)}{n} \leq \theta$. 

Let $x =(x_{n})_{n \in \Z}\in \mathcal{D}(N,\theta)$ be given and let $0 < \gamma < \frac{1}{10}(\tau-\theta)$. Choose points $z^{i}$, $1 \leq i \leq \l$, in the cylinders $[0;i]$ and for $y \in [0;i]$ recall that $[z^{i},y]$ denotes the unique point in the intersection $W^{s}_{loc}(y) \cap W^{u}_{loc}(z_{i})$. We set $\psi^{u}(y) = \psi([z^{i},y])$ for $y \in [0;i]$, $1 \leq i \leq \l$, and note that 
\[
\psi^{u} = \psi + \varphi \circ f - \varphi
\]
where $\varphi(y) = \sum_{k=0}^{\infty}\psi(f^{k}(y))-\psi(f^{k}([z^{i},y]))$ for $y \in [0;i]$. Since $\psi$ is Lipschitz with respect to the distance $\rho$ and $[z^{i},y] \in W^{s}_{loc}(y)$, standard distortion estimates show that the sum defining $\varphi$ converges uniformly and hence $\varphi$ is continuous and bounded on $\Sigma$. We similarly set $S_{n}\psi^{u}(x) = \sum_{j=0}^{n-1}\psi^{u}(f(^{jN}(x)))$ and note that
\[
\left\|\frac{S_{n}\psi^{u} }{n}- \frac{S_{n}\psi}{n}\right\|_{\infty} \leq \frac{2}{n}\|\varphi\|_{\infty}. 
\]
The function $\psi^{u}: \Sigma \rightarrow \R$ is constant on local stable sets and hence descends to a map $\psi^{u}: \Sigma^{u} \rightarrow \R$. We observe that 
\[
\int_{\Sigma^{u}}\psi^{u}\,d\mu^{u} = \int_{\Sigma}\psi\,d\mu < \theta
\]
By Egorov's theorem, given any $\delta > 0$ we can find $K \subset \Sigma^{u}$ with $\mu^{u}(\Sigma \backslash K) > \delta$ and such that $\frac{S_{n}\psi^{u} }{n}$ converges uniformly to $\int_{\Sigma^{u}}\psi^{u}\,d\mu^{u}$ on $K$. Choose $K$ corresponding to a $\delta$ with $\delta < \inf_{1 \leq i \leq \l} \mu([0;i])$; this forces $K \cap [0;i] \neq \emptyset$ for each $1\leq i\leq \l$. 

For each $m \geq 0$ let 
\[
K_{m}^{u}(x) = f_{u}^{-m}(K) \cap [0;x_{0},x_{1},\dots,x_{m}]^{u}
\]
By proposition \ref{jacobians} together with the fact that $K \cap [0;i] \neq \emptyset$ for each $1\leq i\leq \l$ we have $\mu^{u}(K_{m}^{u}(x)) > 0$ for every $m \geq 0$. 

Choose $M$ large enough that for $n \geq M$ we have for every $w \in K$,
\[
\frac{S_{n}\psi^{u}(w)}{n} < \int_{\Sigma^{u}}\psi^{u}\,d\mu^{u}  + \gamma < \theta + \gamma 
\]
Now let $y \in \pi_{u}^{-1}(K_{m}^{u}(x))$ for some $m \geq 0$. For $n \geq 1$ we write 
\[
\frac{S_{n}\psi(y)}{n} = \frac{S_{m}\psi(y)-S_{m}\psi(x)}{n} +  \frac{S_{m}\psi(x) }{n}+\frac{S_{n-m}\psi^{u}(f^{m}(y))}{n}
\]
with the understanding that $S_{n-m}\psi(f^{m}(y)):= 0$ if $n \leq m$. We bound each of the three terms on the right side separately. For the first term we have
\begin{align*}
\frac{S_{m}\psi(y)-S_{m}\psi(x)}{n} &\leq \frac{S_{m}\psi^{u}(y)-S_{m}\psi^{u}(x)}{n} + \frac{2}{n}\|\varphi\|_{\infty} \\
&\leq \frac{\text{Lip}(\psi^{u})}{n}\left(\sum_{k=0}^{m-1}\rho(f^{kN}(x),f^{kN}(y))\right) + \frac{2}{n}\|\varphi\|_{\infty}  \\
&\leq \frac{1}{n}\left(\frac{\text{Lip}(\psi^{u})}{1-e^{-\tau}}+ 2\|\varphi\|_{\infty}\right)
\end{align*}
where $\text{Lip}(\psi^{u})$ denotes the Lipschitz constant of $\psi^{u}$. For the second term we use the bound 
\[
\frac{S_{m}\psi(x) }{n} \leq \left\{ \begin{array}{cc} \frac{N}{n} \|\psi\|_{\infty}, & m \leq N \\
\frac{m}{n}\theta, & m > N \end{array} \right.
\]
Similarly, for the third term we use the bound 
\[
\frac{S_{n-m}\psi^{u}(f^{m}(y)) }{n} \leq \left\{ \begin{array}{cc} \frac{M}{n} \|\psi^{u}\|_{\infty}, & n-m \leq M \\
\frac{n-m}{n}(\theta + \gamma) , & m > N \end{array} \right.
\]
Choose $R$ large enough that 
\[
\frac{1}{R}\left(\frac{\text{Lip}(\psi^{u})}{1-e^{-\tau}}+ 2\|\varphi\|_{\infty}\right) < \gamma
\]
and also large enough that $\frac{N}{R}\|\psi\|_{\infty} < \gamma$ and $\frac{M}{R}\|\psi^{u}\|_{\infty} < \gamma$.  We conclude by combining all of the above estimates that for any $m \geq 1$, $x \in \mathcal{D}(N,\theta)$,  $y \in \pi_{u}^{-1}(K_{m}^{u}(x))$ and $n \geq R$ we have
\[
\frac{S_{n}\psi(y)}{n} < \theta+ 3\gamma.
\]
By our choice of $\gamma$ we have $\theta + 3\gamma < \tau$, so we can set $\theta_{*} = \theta + 3\gamma$. We conclude that for every $m \geq 1$,  $x \in \mathcal{D}(N,\theta)$, $y \in \pi_{u}^{-1}(K_{m}^{u})(x)$ and $k \geq 1$ we have 
\[
\prod_{j = 0}^{k-1} \left\|A^{N R}(f^{ j N R}(y))\right\| \left\|A^{N R}(f^{j N R}(y))^{-1}\right\| \leq e^{N \cdot S_{kR}\psi(y)} \leq e^{k N R\theta_{*}}.
\]
We may replace $f$ and $A$ everywhere in the above proof by $f^{-1}$ and $A^{-1}$ and run the same arguments again with $u$ everywhere replaced by $s$. We obtain an $R' \geq N$ and an analogous sequence of subsets 
\[
K_{m}^{s}(x) \subset [-m;x_{-m},x_{-m+1},\dots,x_{-1}]^{s},
\] 
of $\Sigma^{s}$ with $\mu^{s}(K_{m}^{s}(x)) > 0$ for every $m \geq 1$ such that for $m \geq 1$, $x \in \mathcal{D}(N,\theta)$, $y \in \pi_{s}^{-1}(K_{m}^{s})(x)$ and $k \geq 1$ we have 
\[
\prod_{j = 0}^{k-1} \left\|A^{-N R'}(f^{- j N R'}(y))\right\| \left\|A^{-N R'}(f^{-j N R'}(y))^{-1}\right\| \leq e^{N \cdot S_{k R'}\psi(y)} \leq e^{k N R'\theta_{*}}.
\]
Let $N_{*} = N \cdot \text{max}(R, R')$ and for $m \geq 1$ let 
\[
E_{m}(x) = \pi_{s}^{-1}(K_{m}^{s})(x) \cap \pi_{u}^{-1}(K_{m}^{u})(x)
\]
We conclude that if $x \in \mathcal{D}(N,\theta)$ and $y \in E_{m}(x)$ for some $m \geq 1$ then $y \in \mathcal{D}(N_{*},\theta_{*})$. 

Let $x \in \mathcal{D}(N,\theta)$. From the local product structure of the measure $\mu$ on the cylinder $[0;x_{0}]$ it follows that 
\[
\mu\left(E_{m}(x)\right) \geq \left(\inf_{\Sigma}\xi\right)\mu^{u}(K_{m}^{u}(x)) \mu^{s}(K_{m}^{s}(x)) > 0
\]
for every $m \geq 1$. Furthermore it is easy to check from the construction that $E_{m}(x)$ is contained in the $\rho$-ball of radius $e^{-m\tau}$ centered at $x$ in $\Sigma$. Since $E_{m}(x) \subset \mathcal{D}(N_{*},\theta_{*})$ for every $m \geq 1$ and the $\rho$-balls of radius $e^{-m\tau}$ form a neighborhood basis of $x$ as $m \rightarrow \infty$, we conclude that $x \in \text{supp}(\mu | \mathcal{D}(N_{*},\theta_{*}))$ as desired. 
\end{proof} 

We are now able to construct the desired extension $\hat{\eta}$ of $\eta$ to $\mathcal{D}(N,\theta)$ for each $N > 0$ large enough and $\theta < \tau$. Fix $\theta < \tau$ and take $N$ large enough that for $1 \leq i \leq \l$ we have $\mu(\mathcal{D}(N,\theta) \cap [0;i]) > 0$. Choose points $\omega^{i} \in \Omega\cap \mathcal{D}(N,\theta)  \cap [0;i]$ for $1 \leq i \leq \l$. For any point $x  \in \mathcal{D}(N, \theta) \cap [0;i]$ we then define
\[
\hat{\eta}_{x} = H^{s}_{[\omega^{i}, x] x}H^{u}_{\omega^{i} [\omega^{i}, x] }[\eta_{\omega^{i}}]
\]
\begin{lem}\label{extension}
For each $\theta < \tau$ and $N > 0$, $\hat{\eta}$ is uniformly continuous on $\mathcal{D}(N,\theta)$. If $x \in \mathcal{D}(N,\theta)$ then for any $y \in W^{*}_{loc}(x)$ we have $H^{*}_{xy}[\hat{\eta}_{x}] = \hat{\eta}_{y}$ for $* \in \{s,u\}$ and 
\[
A(x)[\hat{\eta}_{x}] = \hat{\eta}_{f(x)}
\]
\end{lem}

\begin{proof}
Fix $N > 0$ and $\theta < \tau$ as in the construction of $\hat{\eta}$ above. Apply Lemma \ref{Lipschitzapprox} to $N$ and $\theta$ to obtain $\theta_{*} < \tau$  and $N_{*}$ such that  
\[
\mathcal{D}(N,\theta) \subseteq \text{supp}(\mu | \mathcal{D}(N_{*},\theta_{*})).
\]
Since $\eta$ is uniformly continuous on $\Omega \cap \mathcal{D}(N_{*},\theta_{*})$ and this subset has full measure in $\mathcal{D}(N_{*},\theta_{*})$, we conclude that $\eta$ has a unique continuous extension to the set $\text{supp}(\mu | \mathcal{D}(N_{*},\theta_{*}))$ and therefore a unique continuous extension to the set $\mathcal{D}(N,\theta)$. 

Since every point of $\text{supp}(\mu | \mathcal{D}(N_{*},\theta_{*}))$ is a limit point of the set $\Omega \cap \mathcal{D}(N_{*},\theta_{*})$ on which $\eta$ is invariant under the stable and unstable holonomies $H^{s}$ and $H^{u}$ by Lemma \ref{holonomyinvariance} and these holonomies are uniformly continuous on $\mathcal{D}(N_{*},\theta_{*})$ we conclude that this continuous extension of $\eta$ to $\text{supp}(\mu | \mathcal{D}(N_{*},\theta_{*}))$ is invariant under $H^{u}$ and $H^{s}$. Since $\omega^{i} \in \text{supp}(\mu | \mathcal{D}(N_{*},\theta_{*}))$ for $1 \leq i \leq \l$ this implies that this continuous extension coincides with $\hat{\eta}$. Thus $\hat{\eta}$ is uniformly continuous on $\mathcal{D}(N,\theta)$ and invariant under both $H^{s}$ and $H^{u}$ for any choice of $N > 0$, $\theta < \tau$. 

Lastly we need to show that $\hat{\eta}$ is $A$-invariant. To do this we use the following easy proposition, 

\begin{prop}\label{gap}
Let $\e > 0$ be given. Then for every $N$ large enough we have for any $\theta > 0$, 
\[
f^{-1}(\mathcal{D}(N,\theta)) \subseteq \mathcal{D}(N,\theta+\e).
\]
\end{prop}

\begin{proof}
For every $x \in \Sigma$ and $N > 0$ we have the inequality
\[
\|A^{N}(f^{-1}(x))\| \leq \|A^{N}(x)\| \cdot \|A(f^{N}(x))^{-1}\| \cdot \|A(f^{-1}(x))\|
\]
Let $R = \sup_{x \in \Sigma} \text{max}\{\|A(x)\|,\|A(x)^{-1}\|\}$. Then we have 
\[
\|A^{N}(f^{-1}(x))\| \leq R^{2}\|A^{N}(x)\| 
\]
and similarly 
\[
\|A^{-N}(f^{-1}(x))\| \leq R^{2}\|A^{-N}(x)\| 
\]
for every $x \in \Sigma$. We conclude that for every $x \in \mathcal{D}(N,\theta)$, 
\[
\prod_{j = 0}^{k-1} \|A^{N}(f^{jN}(f^{-1}(x)))\| \|A^{N}(f^{jN}(f^{-1}(x)))^{-1}\| \leq R^{4k}e^{kN\theta} \; \; \text{for all $k \geq 1$},
\]
and the same inequality holds with $f^{-1}$ and $A^{-1}$ replacing $f$ and $A$. Hence we see that if $R^{4}\leq e^{N\e}$ then $f^{-1}(x) \in  \mathcal{D}(N,\theta+\e)$, and this clearly holds for $N$ large enough. 
\end{proof}  

By Proposition \ref{gap} we can (by taking $N_{*}$ larger if necessary) find some $0 < \e < \tau - \theta_{*}$ such that 
\[
f^{-1}(\mathcal{D}(N_{*},\theta_{*})) \subset  \mathcal{D}(N_{*},\theta_{*}+\e)
\]
Applying the conclusions of the previous paragraph to $\mathcal{D}(N_{*},\theta_{*}+\e)$ since $\theta_{*}+\e < \tau$, we conclude that $\hat{\eta}$ is uniformly continuous on $\mathcal{D}(N_{*},\theta_{*}+\e)$ and thus $x \rightarrow \hat{\eta}_{f^{-1}(x)}$ is a uniformly continuous function on $\mathcal{D}(N_{*},\theta_{*})$. Setting $\zeta_{x} = A(f^{-1}(x))[\hat{\eta}_{f^{-1}(x)}]$, we conclude that  $\zeta$ is uniformly continuous on $\mathcal{D}(N_{*},\theta_{*})$ and $\zeta_{x} = \eta_{x}$ for $\mu$-a.e. $x \in \Sigma$. Hence $\zeta$ also gives a continuous extension of $\eta$ to $\mathcal{D}(N,\theta)$. By the uniqueness of this extension we conclude that $\zeta = \hat{\eta}$ on $\mathcal{D}(N,\theta)$ for any $N > 0$, $\theta <\tau$. This then implies that for $x \in \mathcal{D}(N,\theta)$, 
\[
\hat{\eta}_{x} = A(f^{-1}(x))[\hat{\eta}_{f^{-1}(x)}],
\]
which is equivalent to the desired invariance property of $\hat{\eta}$. 
\end{proof}

\subsection{Vanishing of periodic exponents}\label{subsec: vanishing} We recall that for a periodic point $p$ we denote the $f$-invariant probability measure supported on the orbit of $p$ by $\nu_{p}$ and we denote the extremal Lyapunov exponents of $\nu_{p}$ by $\la_{+}(p)$ and $\la_{-}(p)$. The following lemma marks a major step in the proof of Proposition \ref{periodiczero} which comes as a consequence of our work in Section \ref{subsec: extension}. We are able to rule out the existence of periodic points for which the gap between the two extremal Lyapunov exponents of $\mathcal{A}$ falls below a certain threshold. We are also able to find a periodic point $q$ at which $\la_{+}(q) = 0$; in fact the techniques of the proposition make it clear how to produce many such points, but we will only need one for our purposes.

\begin{lem}\label{smallperiodiczero} Suppose that $\mathcal{A}$ preserves a measurable conformal structure. Then, 
\begin{enumerate}
\item There are no periodic points $p$ for which the inequality,
\[
0 < \la_{+}(p)-\la_{-}(p) < \tau,
\]
 holds. 
 \item There is a periodic point $q$ such that $\la_{+}(q) = 0$. 
 \end{enumerate}
\end{lem}

\begin{proof}
Suppose we have a periodic point $p$ of $f$ which satisfies the inequality  $\la_{+}(p)-\la_{-}(p) < \tau$. Choose $\theta$ satisfying the inequality 
\[
\la_{+}(p)-\la_{-}(p) < \theta < \tau.
\]
We will show that there is an $N > 0$ such that $p \in \mathcal{D}(N,\theta)$. 

Let $k$ be the period of $p$. As a direct consequence of the defining formulas for the extremal Lyapunov exponents of $p$ we have 
\[
\la_{+}(p) - \la_{-}(p)  =  \lim_{n \rightarrow \infty} \frac{\log \|A^{kn}(p)\| \|(A^{kn}(p))^{-1}\|}{kn}.
\]
Chooose $M$ large enough that for $n \geq M$,
\[
\|A^{kn}(p)\| \|(A^{kn}(p))^{-1}\| \leq e^{kn\theta}.
\] 
Let $N$ = $k M$. Then, using the periodicity of $p$, for each $s \geq 1$ we have, 
\begin{align*}
\prod_{j = 0}^{s-1} \|A^{N}(f^{jN}(p))\| \|(A^{N}(f^{jN}(p)))^{-1}\| &= \prod_{j = 0}^{s-1} \|A^{N}(f^{jkM}(p))\| \|(A^{N}(f^{jkM}(p)))^{-1}\| \\
&= (\|A^{N}(p)\| \|(A^{N}(p))^{-1}\|)^{s} \\
&\leq e^{sN\theta}
\end{align*}
Since $A^{-kM}(p) = (A^{kM}(p))^{-1}$ and $(A^{-kM}(p))^{-1} = A^{kM}(p)$ we also have 
\[
\prod_{j = 0}^{s-1} \|A^{-N}(f^{-jN}(p))\| \|(A^{-N}(f^{-jN}(p)))^{-1}\| \leq e^{sN\theta}
\]
Hence $p \in \mathcal{D}(N,\theta)$. 

Since $\theta < \tau$, we conclude by Lemma \ref{extension} that the extension $\hat{\eta}$ of $\eta$ is defined at $p$ and satisfies $A^{k}(p)[\hat{\eta}_{p}] = \hat{\eta}_{p}$. Thus $A^{k}(p)$ preserves an inner product on $\R^{d}$; we conclude that $\la_{+}(p) = 0$. This proves the first part of the lemma. 

By Kalinin's exponent approximation theorem \cite{Kal} there is at least one periodic point $q$ such that $\la_{+}(q)-  \la_{-}(q)< \tau$. This is because we have an ergodic $f$-invariant probability measure $\mu$ for which $\la_{+}(\mathcal{A},\mu) =  0$ and $\la_{-}(\mathcal{A},\mu) = 0$. By the first part of the lemma we must then have $\la_{+}(q) - \la_{-}(q) = 0$. This implies that $\la_{+}(q) = 0$ since $-\la_{-}(q) > 0$ (because $\la_{+}(x) > 0$ and $\mathcal{A}$ is an $SL(d,\R)$-valued cocycle). 
\end{proof}

A prospective strategy for finishing the proof is to show that if there is some periodic point $x \in \Sigma$ with $\la_{+}(x) - \la_{-}(x) \geq \tau$ then, letting $y$ be a periodic point such that $\la_{+}(y) = 0$, we can produce a new periodic point $p$ shadowing some mixture of the orbits of $x$ and $y$ such that $0 < \la_{+}(p) - \la_{-}(p) < \tau$. This would contradict Lemma \ref{smallperiodiczero}.  

However the quantity $\la_{+}(p) - \la_{-}(p)$ is difficult to estimate from below; this is tied to the inherent difficulty of proving that cocycles have nonzero exponents. We will take an alternative approach which does not require us to directly estimate the Lyapunov exponents of the periodic point $p$ that we construct. 

\begin{proof}[Proof of Proposition \ref{periodiczero}]
We assume to a contradiction that there is a periodic point $x$ of $f$ with $\la_{+}(x) > 0$.  Let $y$ be a periodic point given by Lemma \ref{smallperiodiczero} such that $\la_{+}(y) = 0$.   Choose $k$ such that both $x$ and $y$ are fixed by $f^{k}$. For each positive integer $m$ which is divisible by $k$, we will construct a new periodic point $p^{m}$ symbolically using orbit segments of $f$ which shadow the orbits of $x$ and $y$ and which satisfies $f^{u_{m}}(p^{m}) = p^{m}$ for a certain positive integer $u_{m}$. 

We will show that we can construct this sequence of periodic points $\{p^{m}\}$ such that the following two statements must hold simultaneously: First, we claim that as $m \rightarrow \infty$ we have $\|A^{u_{m}}(p^{m})\| \rightarrow \infty$. Second, we claim that there is a positive integer $N$ and a $\theta < \tau $ such that $p^{m} \in \mathcal{D}(N,\theta)$ for all $m$ large enough. 

This produces a contradiction: by Lemma \ref{extension} $\hat{\eta}$ is defined and uniformly continuous on the compact set $\mathcal{D}(N,\theta)$.  There is thus a constant $\gamma > 0$ such that for any vector $v \in \R^{d}$ and any  $q \in \mathcal{D}(N,\theta)$ we have 
\[
\gamma^{-1}\|v\|^{2} \leq \hat{\eta}_{q}(v,v) \leq \gamma \|v\|^{2}
\]
Since $A^{u_{m}}(p^{m})[\hat{\eta}_{p^{m}}] = \hat{\eta}_{p^{m}}$, we thus have 
\begin{align*}
\|A^{u_{m}}(p^{m})(v)\|^{2} &\leq \gamma \cdot \hat{\eta}_{q}(A^{u_{m}}(p^{m})(v),A^{u_{m}}(p^{m})(v)) \\
&= \gamma \cdot \hat{\eta}_{q}(v,v) \\
&\leq \gamma^{2} \|v\|^{2}
\end{align*}
Thus $\|A^{u_{m}}(p^{m})(v)\| \leq \gamma$ for all large enough $m$, which contradicts the previous assertion that $\|A^{u_{m}}(p^{m})\| \rightarrow \infty$ as $m \rightarrow \infty$.

We write $x = (x_{n})_{n \in \Z}$, $y = (y_{n})_{n \in \Z}$ in the shift coordinates on $\Sigma$. Recall that $Q$ denotes the matrix defining valid words for the subshift of finite type $\Sigma$. Since $\Sigma$ is a topologically mixing subshift of finite type there is an $M > 0$ such that for all $n \geq M$ the matrix  $Q^{n}$ has all positive entries, meaning that any two letters in the alphabet $\{1, \dots, \l\}$ can be connected by a valid word of length $n$.  By increasing $k$ if necessary we may assume that $k \geq M$ and consequently $m \geq M$ since $m$ is a multiple of $k$. Let $b$ and $c$ be two positive integers whose values will be tuned later. 

Set $u_{m} = (2b + c + 1)m$. Recall that we restricted $m$ to be a multiple of $k$ so that $f^{m}(x) = x$ and $f^{m}(y) = y$. We define $p^{m} = (p^{m}_{n})_{n \in \Z}$ where the numbers $p^{m}_{n} \in \{1, \dots,\l\}$ are chosen as follows: for $-bm \leq j \leq bm$ we set $p^{m}_{j} = y_{j}$. For $(b+1)m \leq j \leq (b+ c + 1)m$ we define $p^{m}_{j} = x_{j-(b+1)m}$. 

For $(b+1)m  < j < (b+ c + 2)m $ we let the string $(p_{(b+1)m}^{m},\dots,p_{(b+ c +1)m}^{m})$ be any valid word in the alphabet of $\Sigma$ which is of length $m$ and connects $y_{m}$ to $x_{0}$. This is possible because we chose $m$  such that $m \geq M$. Similarly for $(b+ c +1)m \leq  j \leq (b+ c +2)m$ we can fill in $m$ coordinates for $p^{m}_{j}$ which connect $x_{ (b+ c +1)m}$ to $y_{0}$ in the subshift $\Sigma$. We then define the rest of the coordinates of $p^{m}$ by declaring $p^{m}$ to be periodic with period $u_{m}$. 

Let $\e > 0$ be a small parameter which will be tuned later along with $b$ and $c$. For ease of notation in what follows we set $\la:= \la_{+}(x)$ and $\zeta = \sup_{z \in \Sigma}  \log  \|A(z)\| \|A(z)^{-1}\|$. We set $\chi := c(\la - \e) - 2b\e - 2\zeta$ and fix some $\theta$ with $0 < \theta < \tau$ (for example $\theta = \tau/2$). 

We first show that if $\chi > 0$ then $\|A^{u_{m}}(p^{m})\| \rightarrow \infty$ as $m \rightarrow \infty$. We will address later the question of how to choose $b$, $c$ and $\e$ together to guarantee that $\chi > 0$ while also satisfying some other necessary conditions. We will need the following lemma, 

\begin{lem}\label{periodicapprox}
Let $p \in \Sigma$ be a periodic point for $f$ with extremal Lyapunov exponent $\la :=\la_{+}(p)$ and let $\e > 0$ be given. Then there is a constant $C \geq 1$ independent of $n$ such that for every $n \geq 1$, if $q \in \Sigma$ satisfies 
\[
\rho(f^{j}(p),f^{j}(q)) \leq \text{max}\{e^{-j\tau},e^{-(n-j)\tau}\}, \; 0 \leq j \leq n,
\]
then 
\[
C^{-1}e^{n(\la-\e)}\leq \|A^{n}(q)\| \leq  Ce^{n(\la+\e)}.
\]
\end{lem} 

\begin{proof}
This lemma follows from \cite[Lemma 3.1]{Kal} together with the observation that for a probability measure $\nu_{p}$ supported on a periodic point $p$, there is a uniform comparison between the $\e$-Lyapunov norm and the Euclidean norm which depends only on $\e$. Lemma 3.1 of \cite{Kal} then gives that there is a $\delta > 0$ and a constant $C' \geq 1$ such that for every $n \geq 1$, if 
\[
\rho(f^{j}(p),f^{j}(q)) \leq \delta\cdot \text{max}\{e^{-j\tau},e^{-(n-j)\tau}\}, \; 0 \leq j \leq n,
\]
then
\[
C'^{-1}e^{n(\la-\e)}\leq \|A^{n}(q)\| \leq  C'e^{n(\la+\e)}.
\]
To obtain our desired statement, fix $m$ large enough that $e^{-m\tau} < \delta$. Then for every $n \geq 1$, if 
\[
\rho(f^{j}(p),f^{j}(q)) \leq\text{max}\{e^{-j\tau},e^{-(n-j)\tau}\}, \; 0 \leq j \leq n,
\]
then 
\[
\rho(f^{j}(p),f^{j}(q)) \leq \delta \cdot \text{max}\{e^{-j+m\tau},e^{-(n-j+m)\tau}\} , \; m\leq j \leq n-m,
\]
and consequently 
\[
C'^{-1}e^{(n-2m)(\la-\e)}\leq \|A^{n-m}(f^{m}(q))\| \leq  C'e^{(n-2m)(\la+\e)}.
\]
We conclude that 
\[
C^{-1}e^{n(\la-\e)}\leq \|A^{n}(q)\| \leq  Ce^{n(\la+\e)},
\]
with $C = C' e^{2m(\la+\e)}\sup_{z \in \Sigma}\|A^{m}(z)\|\|A^{m}(z)^{-1}\|$.
\end{proof}

For $0 \leq j \leq bm$ we have
\[
\rho(f^{j}(y),f^{j}(p^{m})) \leq \text{max}\{e^{-j \tau}, e^{-(bm-j)\tau}\}.
\]
For $0 \leq j \leq cm$, 
\[
\rho(f^{j + (b+1)m}(x),f^{j + (b+1)m}(p^{m})) \leq \text{max}\{e^{-j \tau}, e^{-(cm-j)\tau}\}. 
\]
Finally for $0 \leq j \leq bm$
\[
\rho(f^{j+(b+c+2)m}(y),f^{j+(b+c+2)m}(p^{m})) \leq \text{max}\{e^{-j \tau}, e^{-(bm-j)\tau}\}.
\]

Apply Lemma \ref{periodicapprox} to the periodic points $x$ and $y$ with the $\e > 0$ chosen above and let $C$ be the maximum of the two constants in the output of Lemma \ref{periodicapprox} for $x$ and $y$ respectively. 

We have the lower bound using Lemma \ref{periodicapprox},
\begin{align*}
\|A^{u_{m}}(p^{m})\| &\geq \|A^{cm}(f^{(b+1)m}(p^{m}))\| \cdot \|A^{bm}(p^{m})^{-1}\|^{-1} \cdot \|A^{km}(f^{bm}(p^{m}))^{-1}\|^{-1} \\
&\cdot \|A^{m}(f^{(b+c+1)m}(p^{m}))^{-1}\|^{-1} \cdot \|A^{bm}(f^{(b+c+1)m}(p^{m}))^{-1}\|^{-1} \\
&\geq C^{-3}\exp((c(\la - \e) - 2b\e - 2\zeta)m) \\
&= C^{-3}\exp(\chi m) 
\end{align*}
If $\chi > 0$ then it follows that $\|A^{u_{m}}(p^{m})\| \rightarrow \infty$ as $m \rightarrow \infty$. 

It remains to prove the claim that there is an $N > 0$ such that for every $m \geq 1$ large enough we have $p^{m} \in \mathcal{D}(N,\theta)$. We will first show that there is an $N >0$ such that for $m \geq 1$ large enough and every $ s \geq 1$, 
\[
\prod_{j=0}^{s-1} \|A^{N}(f^{jN}(p^{m}))\| \|A^{N}(f^{jN}(p^{m}))^{-1}\| \leq e^{s N \theta}.
\]
We will prove this first set of inequalities using the fact that upon applying $f$ the periodic point $p^{m}$ first shadows the orbit of $y$ for $bm$ iterates, then the orbit of $x$ for $cm$ iterates, then again shadows the orbit of $y$ for another $bm$ iterates (with $2m$ transitioning iterates where $f$ does not necessarily shadow $x$ or $y$).

The points $p^{m}$ are constructed such that if we apply $f^{-1}$ we get the same pattern of shadowing $y$ for $bm$ iterates, $x$ for $cm$ iterates, then $y$ again for $bm$ iterates. Hence it can be checked that the same proof given for the first set of inequalities will also prove the second set of inequalities, 
\[
\prod_{j=0}^{s-1} \|A^{-N}(f^{-jN}(p^{m}))\| \|A^{-N}(f^{-jN}(p^{m}))^{-1}\| \leq e^{s N \theta}. 
\]
and thus establish that $p^{m} \in \mathcal{D}(N,\theta)$ for all $m$ large enough. 

We proceed with the proof of the first set of inequalities. The main idea of the proof is simple: the ratio $\frac{b}{c}$ controls the ratio of the number of iterates that $p^{m}$ shadows $y$ to the number of iterates that $p^{m}$ shadows $x$. We essentially showed above that if $c$ is large enough then the influence of the positive exponent of $x$ forces $\|A^{u_{m}}(p^{m})\| \rightarrow \infty$ as $m \rightarrow \infty$. We will show that once $b$ is chosen such that the ratio $\frac{b}{c}$ is large enough then the time that $p^{m}$ spends shadowing the point $y$ is sufficient to cancel enough of the influence of the Lyapunov exponents of $x$ to guarantee that $p^{m}$ lands in $\mathcal{D}(N,\theta)$ for some choice of $N$ that is independent of $m$. 


We now specify the parameters $b$, $c$, and $\e$. Recall that $\theta$ denotes a fixed number satisfying $0 < \theta < \tau$ and that we defined $\la = \la_{+}(x)$ and $\zeta = \sup_{z \in \Sigma} \log \|A(z)\| \|A(z)^{-1}\|$. We also set $\xi = \la_{+}(x) - \la_{-}(x)$. We choose the integer $c$ large enough that $c \la - 2\zeta > 0$. Having chosen $c$, we now choose $b$ large enough that the following three inequalities hold: first we require that 
\begin{equation}\label{ineq1}
\zeta\left(1-\frac{b}{b+1}\right) + \frac{\theta}{10} < \frac{9}{10}\theta.
\end{equation}
Second we require that 
\begin{equation}\label{ineq2}
\frac{b}{b+1}\cdot \frac{\theta}{10} + \frac{\zeta}{b+1} +\left(1-\frac{b+1}{b+c+1}\right)\left(\xi+\frac{\theta}{10}\right) < \frac{9}{10}\theta.
\end{equation}
Finally we require that 
\begin{equation}\label{ineq3}
\frac{b}{b+ c + 1}\cdot \frac{\theta}{10} + \frac{\zeta}{b+ c + 1} +\frac{c}{b+c+1}\left(\xi+\frac{\theta}{10}\right) +  \zeta \left(1-\frac{b+c+1}{b+c+2}\right)  < \frac{9}{10}\theta.
\end{equation}
It's easily verified that, for a fixed choice of $c$, each of these inequalities holds once $b$ is large enough. Finally, having chosen $b$ and $c$, we now choose $\e$ small enough that $\e \leq \theta/10$ and $\chi = c(\la - \e) - 2b\e  - 2\zeta > 0$. We note that since $\frac{9}{10}\theta \leq \theta - \e$, the three inequalities above also hold with $\theta - \e$ on the right hand side instead. The source of these three inequalities will become clear from the computations below. 

Let $J > 0$ be large enough that for $j \geq J$, 
\[
\|A^{j}(y)\| \|A^{j}(y)^{-1}\| \leq e^{j\e}, \; \|A^{j}(x)\| \|A^{j}(x)^{-1}\| \leq e^{j(\xi + \e)},
\]
Choose $r$ to be a multiple of $k$ which satisfies $r \geq J$. Let $L := \text{Lip}(\log \|A^{r}\| \|(A^{r})^{-1}\|)$ be the Lipschitz constant of $\log \|A^{r}\| \|(A^{r})^{-1}\|$. 
We define a constant
\[
C:= 2\sum_{j=0}^{\infty}e^{-j\tau} = \frac{2}{1-e^{-\tau}}.
\] 

Since we only need a subsequence of periodic points $p^{m_{q}}$ such that there is some $N$ for which $p^{m_{q}} \in \mathcal{D}(N,\theta)$, we can further restrict $m$ to be divisible by $r$ so that $m = qr$ for some integer $q$. For $0 \leq s \leq bq$ we then have
\begin{align*}
\log\left(\prod_{j = 0}^{s-1}\frac{\|A^{r}(f^{jr}(p^{m}))\| \|A^{r}(f^{jr}(p^{m}))^{-1}\|}{\|A^{r}(f^{jr}(y))\| \|A^{r}(f^{jr}(y))^{-1}\|}\right) &\leq L\sum_{j=0}^{s-1}\rho(f^{jr}(p^{m}),f^{jr}(y)) \\
&\leq L \sum_{j=0}^{s-1}\text{max}\{e^{-jN \tau}, e^{-(km-jN)\tau}\} \\
&\leq L C,
\end{align*}
for the constant $C$ which is independent of $m$ and $s$. For $(b+1)q \leq s \leq (b+c+1)q$, we replace $y$ by $x$ in the above estimate 
\begin{align*}
\log\left(\prod_{j = (b+1)q}^{s-1}\right.&\left.\frac{\|A^{r}(f^{jr}(p^{m}))\| \|A^{r}(f^{jr}(p^{m}))^{-1}\|}{\|A^{r}(f^{jr}(x))\| \|A^{r}(f^{jr}(x))^{-1}\|}\right) \\
&\leq \sum_{j= (b+1)q}^{s-1}L \rho(f^{jr}(p^{m}),f^{jr}(x)) \\
&=  \sum_{j= 0}^{s-1-(b+1)q}L \rho(f^{jr +(b+1)q }(p^{m}),f^{jr + (b+1)q}(x))\\
&\leq LC.
\end{align*}
Similar estimates give for $(b+c+2)q \leq s \leq (2b+c+2)q$, 
\[
\log\left(\prod_{j = (b+c+2)q}^{s-1}\frac{\|A^{r}(f^{jr}(p^{m}))\| \|A^{r}(f^{jr}(p^{m}))^{-1}\|}{\|A^{r}(f^{jr}(y))\| \|A^{r}(f^{jr}(y))^{-1}\|}\right) \leq LC
\]
We conclude, using the shadowing of $p^{m}$ along the orbit of $y$ for the first $bm$ iterates, that for $0 \leq s \leq bq$,
\begin{align*}
\prod_{j=0}^{s-1}\|A^{r}(f^{jr}(p^{m}))\| \|A^{r}(f^{jr}(p_{m}))^{-1}\| &\leq e^{LC}\left(\prod_{j= 0}^{s-1}\|A^{r}(f^{jr}(y))\| \|A^{r}(f^{jr}(y))^{-1}\|\right) \\
&\leq \exp(LC + rs\e) \\ 
&\leq \exp(LC + rs(\theta- \e)),
\end{align*}
since $\theta > 2\e$. For $bq \leq s \leq (b+1)q$ we combine the above estimate with the crude bound 
\[
\prod_{j=bq}^{s-1}\|A^{r}(f^{jr}(p^{m}))\| \|A^{r}(f^{jr}(p^{m}))^{-1}\| \leq \exp(\zeta r(s-bq)),
\]
(where we recall that $\zeta = \sup_{z \in \Sigma} \|A(z)\| \cdot \|A(z)^{-1}\|$) to get 
\begin{align*}
\prod_{j=0}^{s-1}\|A^{r}(f^{jr}(p^{m}))\| \|A^{r}(f^{jr}(p^{m}))^{-1}\| &\leq \exp(LC + brq\e + \zeta r(s-bq)) \\
&= \exp\left(LC + rs\left(\zeta\left(1-\frac{bq}{s}\right) + \frac{bq}{s}\e\right)\right) \\
&\leq  \exp\left(LC + rs\left(\zeta\left(1-\frac{b}{b+1}\right) + \e\right)\right) \\
&\leq \exp\left(LC + rs(\theta - \e)\right),
\end{align*}
where we've used here the inequality $\eqref{ineq1}$ together with the fact that $\e \leq \theta/10$. We similarly use the shadowing of $p^{m}$ along the orbit of $x$ for $cm$ iterates and combine this with the estimates above to obtain for $(b+1)q \leq s \leq (b+c+1)q$,
\begin{align*}
\prod_{j=0}^{s-1}\|A^{r}(f^{jr}(p^{m}))\| \|A^{r}(f^{jr}(p^{m}))^{-1}\| &\leq \exp(2LC + brq\e + \zeta rq )\\
&\cdot \prod_{j= (b+1)q}^{s-1}\|A^{r}(f^{jr+M}(x))\| \|A^{r}(f^{jr+M}(x))^{-1}\| \\
&\leq \exp\left(2LC + brq\e + \zeta rq + r\left(s-(b+1)q\right)(\xi+\e)\right) \\
&\leq \exp\left(2LC +s(\theta - \e) \right)
\end{align*}
To obtain the final inequality above we use the fact that $(b+1)q \leq s \leq (b+c+1)q$ to write 
\begin{align*}
 brq\e + \zeta rq + r\left(s-(b+1)q\right)(\xi+\e)&= rs\left(\frac{bq}{s}\e + \zeta\frac{q}{s} +\left(1-\frac{(b+1)q}{s}\right)(\xi+\e)\right) \\
 &\leq rs\left(\frac{b}{b+1}\e + \frac{\zeta}{b+1} +\left(1-\frac{b+1}{b+c+1}\right)(\xi+\e)\right),
\end{align*}
and from the last line we use the inequality \eqref{ineq2} together with the fact that $\e \leq \theta/10$ to obtain the desired inequality. 
For $(b+c + 1)q \leq s \leq (b+c +2)q$ we again use the crude bound 
\[
\prod_{j=(b+c+1)q}^{s-1}\|A^{r}(f^{jr}(p^{m}))\| \|A^{r}(f^{jr}(p^{m}))^{-1}\| \leq \exp(\zeta r(s-(b+c+1)q)),
\]
to now obtain (combining also with the above estimates) 
\begin{align*}
\prod_{j=0}^{s-1}\|A^{r}(f^{jr}(p^{m}))\| \|A^{r}(f^{jr}(p^{m}))^{-1}\| &\leq \exp(2LC + brq\e + \zeta rq \\
&+ rcq(\xi+\e) + \zeta r(s-(b+c+1)q)),
\end{align*}
We must prove again that the quantity in the exponent on the right is bounded above by $2LC + rs(\theta - \e)$. This involves an application of the inequality \eqref{ineq3} in a manner similar to the previous verification. Using $(b+c + 1)q \leq s \leq (b+c +2)q$ we have 
\begin{align*}
 brq\e + \zeta rq &+ rcq(\xi+\e) + \zeta r(s-(b+c+1)q) \\
 &= rs\left(\frac{bq}{s}\e + \zeta\frac{q}{s} +\frac{cq}{s}(\xi+\e) + \zeta \left(1-\frac{(b+c+1)q}{s}\right) \right) \\
 &\leq \left(\frac{b}{b+ c + 1}\e + \frac{\zeta}{b+ c + 1} +\frac{c}{b+c+1}(\xi+\e) +  \zeta \left(1-\frac{b+c+1}{b+c+2}\right) \right)
\end{align*}
The desired bound then comes from combining the inequality \eqref{ineq3} with the fact that $\e \leq \theta/10$. Finally for $(b+c+2)q \leq s \leq (2b+c+2)q$ we use the approximation by $y$ again and combine this with the previous estimates to obtain 
\begin{align*}
\prod_{j = 0}^{s-1}\|A^{r}(f^{jr}(p^{m}))\| &\|A^{r}(f^{jr}(p^{m}))^{-1}\| \\
&\leq \exp(3LC + (b+c+2)q(\theta - \e)  + r(s-(b+c+2)q)\e) \\
&\leq \exp(3LC + rs(\theta - \e) ),
\end{align*}
once again because $\theta > 2\e$. 

Define the integer $v_{m}$ by $v_{m} := \frac{u_{m}}{r}$. Since $f^{u_{m}}(p^{m}) = p^{m}$, we can extend our arguments to $s \geq (2b+c+2)q = v_{m}$ using the periodicity of $p^{m}$: for $s \geq v_{m}$ we write $s = h v_{m} + s'$ for some positive integers $h$ and $s'$. Then using the above arguments,
\begin{align*}
\prod_{j = 0}^{s-1}\|A^{r}(f^{jr}(p^{m}))\| \|A^{r}(f^{jr}(p^{m}))^{-1}\| &= \left(\prod_{j = 0}^{v_{m}-1}\|A^{r}(f^{jr}(p^{m}))\| \|A^{r}(f^{jr}(p^{m}))^{-1}\| \right)^{h} \\
&\cdot \prod_{j = 0}^{s'-1}\|A^{r}(f^{jr}(p^{m}))\| \|A^{r}(f^{jr}(p^{m}))^{-1}\| \\
&\leq \exp(3LC + rs(\theta-\e))
\end{align*}
Hence we have succeeded in showing that for every $s \geq 1$, 
\[
\prod_{j = 0}^{s-1}\|A^{r}(f^{jr}(p^{m}))\| \|A^{r}(f^{jr}(p^{m}))^{-1}\| \leq \exp(3LC + rs(\theta-\e)).
\]
We use one final trick: let $N = rt$ for some integer $t$ chosen large enough that $N\e > 3LC$. Then it's easy to see that as a consequence of the above inequality we also have 
\[
\prod_{j = 0}^{s-1}\|A^{N}(f^{jN}(p^{m}))\| \|A^{N}(f^{jN}(p^{m}))^{-1}\| \leq \exp(3LC + Ns(\theta-\e)),
\]
for every $s \geq 1$. But we chose $N$ such that  
\[
3LC + Ns(\theta-\e) < N\e + Ns(\theta-\e) \leq Ns\theta
\]
Thus we conclude that $p^{m} \in \mathcal{D}(N,\theta)$ for all large enough $m$. 
\end{proof}

\subsection{Conclusion}\label{conclusion} By Proposition \ref{periodiczero}, $\la_{+}(p) = 0$ for every periodic point $p$ of $f$.  We conclude by \cite[Theorem 1.4]{Kal} that for every $f$-invariant ergodic probability measure $\nu$ on $\Sigma$ we must have $\la_{+}(\mathcal{A},\nu) = -\la_{-}(\mathcal{A},\nu) = 0$. We then apply \cite[Proposition 4.11]{KS} with $\e < \tau$ to obtain a uniform estimate 
\[
\|A^{n}(x)\| \|A^{n}(x)^{-1}\| \leq Ce^{\e |n|}, \; \; \text{for every $x \in \Sigma$, $n \in \Z$},
\] 
for some constant $C \geq 1$. We conclude that there is $N > 0$, $\theta < \tau$ such that $ \mathcal{D}(N,\theta) = \Sigma$. Hence the continuous extension $\hat{\eta}$ of $\eta$ from Lemma \ref{extension} is the desired continuous family of inner products invariant under $\mathcal{A}$. Since $\hat{\eta}$ is invariant under the stable and unstable holonomies $H^{s}$ and $H^{u}$ on $\mathcal{D}(N,\theta) = \Sigma$ and by Proposition \ref{existenceholonomies} these holonomies are uniformly Lipschitz, we conclude that $\hat{\eta}$ is Lipschitz continuous with respect to the distance $\rho$, as desired.

\bibliographystyle{plain}
\bibliography{measurableconformalrigidity}
\end{document}